\documentclass{lms}
\usepackage{amsmath,amssymb,amsfonts} 
\usepackage{graphics,graphicx}                 
\usepackage{color}                    
\usepackage{hyperref,fancyhdr}                 
\usepackage[all]{xy}
\usepackage[applemac]{inputenc}
\usepackage{url}

\newtheorem{theorem}{Theorem}[section]
\newtheorem{proposition}[theorem]{Proposition}
\newtheorem{corollary}[theorem]{Corollary}
\newtheorem{lemma}[theorem]{Lemma}
\newnumbered{definition}[theorem]{Definition}
\newnumbered{example}[theorem]{Example}
\newnumbered{remark}[theorem]{Remark}
\newnumbered{convention}[theorem]{Convention}
\newnumbered{construction}[theorem]{Construction}
\newnumbered{notation}[theorem]{Notation}
\newnumbered{open problem}[theorem]{Open Problem}

\newunnumbered{theo}{Theorem}
\newnumbered{prop}{Proposition}
\newunnumbered{defi}{Definition}
\DeclareMathOperator{\Aut}{Aut}

\DeclareMathOperator{\im}{im}
\DeclareMathOperator{\id}{id}
\DeclareMathOperator{\ev}{ev}

\DeclareMathOperator{\M}{M}

\DeclareMathOperator{\Hom}{Hom}

\DeclareMathOperator{\Diff}{Diff}

\DeclareMathOperator{\SU}{SU}

\DeclareMathOperator{\spec}{spec}

\DeclareMathOperator{\pr}{pr}

\DeclareMathOperator{\Det}{Det}
\DeclareMathOperator{\supp}{supp}

\title[On Noncommutative Principal Torus Bundles]
 {\bf A Geometric Approach to\\Noncommutative Principal Torus Bundles} 

\author{Stefan Wagner}

\classno{46L87, 55R10 (primary), 37B05, 17A60 (secondary)}

\extraline{We thank the \emph{Studienstiftung des deutschen Volkes} for a doctoral scholarship supporting this work.}

\begin{document}

\maketitle

\begin{abstract}
A (smooth) dynamical system with transformation group $\mathbb{T}^n$ is a triple $(A,\mathbb{T}^n,\alpha)$, consisting of a unital locally convex algebra $A$, the $n$-torus $\mathbb{T}^n$ and a group homomorphism \mbox{$\alpha:\mathbb{T}^n\rightarrow\Aut(A)$}, which induces a (smooth) continuous action of $\mathbb{T}^n$ on $A$. In this paper we present a new, geometrically oriented approach to the noncommutative geometry of principal torus bundles based on such dynamical systems. Our approach is inspired by the classical setting: In fact, after recalling the definition of a trivial noncommutative principal torus bundle, we introduce a convenient (smooth) localization method for noncommutative algebras and say that a dynamical system $(A,\mathbb{T}^n,\alpha)$ is called a noncommutative principal $\mathbb{T}^n$-bundle, if localization leads to a trivial noncommutative principal $\mathbb{T}^n$-bundle. We prove that this approach extends the classical theory of principal torus bundles and present a bunch of (non-trivial) noncommutative examples.
\end{abstract}


\thispagestyle{empty}

\tableofcontents

\section{Introduction}

The correspondence between geometric spaces and commutative algebras is a familiar and basic idea of algebraic geometry. Noncommutative Topology started with the famous Gelfand-Naimark Theorems: Every commutative C*-algebra is the algebra of continuous functions vanishing at infinity on a locally compact space and vice versa. In particular, a noncommutative C*-algebra may be viewed as ``the algebra of continuous functions vanishing at infinity'' on a ``quantum space''. The aim of Noncommutative Geometry is to develop the basic concepts of Topology, Measure Theory and Differential Geometry in algebraic terms and then to generalize the corresponding classical results to the setting of noncommutative algebras. The question whether there is a way to translate the geometric concept of a fibre bundle to Noncommutative Geometry is quite interesting in this context. In the case of vector bundles a refined version of the Theorem of Serre and Swan \cite{Swa62} gives the essential clue: The category of vector bundles over a manifold $M$ is equivalent to the category of finitely generated projective modules over $C^{\infty}(M)$. It is therefore reasonable to consider finitely generated projective modules over an arbitrary algebra as ``noncommutative vector bundles". The case of principal bundles is so far not treated in the same satisfactory way. From a geometrical point of view it is not sufficiently well understood what a ``noncommutative principal bundle" should be. Still, there are several approaches towards the noncommutative geometry of principal bundles: For example, there is a well-developed abstract algebraic approach known as Hopf--Galois extensions which uses the theory of Hopf algebras (cf. \cite{Sch04} or \cite[Chapter VII]{Haj04}). Another topologically oriented approach can be found in \cite{ENOO09}; here the authors use $C^*$-algebraic methods to develop a theory of principal noncommutative torus bundles based on \mbox{Green' s} Theorem (cf. \cite[Corollary 15]{Green77}). Furthermore, the authors of \cite{BHMS07} introduce $C^*$-algebraic analogs of freeness and properness, since by a classical result (of Differential Geometry) having a free and proper action of a Lie Group $G$ on a manifold $P$ is equivalent saying that $P$ carries the structure of a principal bundle with structure group $G$. In \cite{Wa11a} we have developed a geometrically oriented approach to the noncommutative geometry of principal bundles based on dynamical systems and the representation theory of the corresponding transformation groups.\\

As is well-known from classical Differential Geometry, the relation between locally and globally defined objects is important for many constructions and applications. For example, a principal bundle $(P,M,G,q,\sigma)$ can be considered as a geometric object that is glued together from local pieces which are trivial, i.e., which are of the form $U\times G$ for some open subset $U$ of $M$. Thus, a natural step towards a geometrically oriented theory of ``noncommutative principal torus bundles" is to describe the \emph{trivial} objects first, i.e., to determine and to classify the trivial noncommutative principal torus bundles. This was done in \cite{Wa11b}:

\begin{definition}\label{Trivial NCP T^n-Bundles}
(Trivial noncommutative principal torus bundles). A (smooth) dynamical system $(A,\mathbb{T}^n,\alpha)$ is called a (\emph{smooth}) \emph{trivial noncommutative principal $\mathbb{T}^n$-bundle}, if each isotypic component $A_{\bf k}$, ${\bf k}\in\mathbb{Z}^n$, contains an invertible element.
\end{definition}

\noindent
This definition is inspired by the following observation: A principal $\mathbb{T}^n$-bundle $(P,M,\mathbb{T}^n,q,\sigma)$ is trivial if and only if it admits a trivialization map. Such a trivialization map consists basically of $n$ smooth functions $f_i:P\rightarrow\mathbb{T}$ satisfying $f_i(\sigma(p,z))=f_i(p)\cdot z_i$ for all $p\in P$ and $z\in\mathbb{T}^n$. From an algebraical  point of view this condition means that each isotypic component of the (naturally) induced dynamical system $(C^{\infty}(P),\mathbb{T}^n,\alpha)$ contains an invertible element. Conversely, each trivial noncommutative principal $\mathbb{T}^n$-bundle of the form $(C^{\infty}(P),\mathbb{T}^n,\alpha)$ induces a trivial principal $\mathbb{T}^n$-bundle of the form $(P,P/\mathbb{T}^n,\mathbb{T}^n,\pr,\sigma)$. The crucial point here is to verify the freeness of the induced action of $\mathbb{T}^n$ on $P$. An important class of examples, which will also show up in this paper, is provided by the so-called noncommutative tori:

\begin{example}\label{NC n-tori as NCPTB}
(a) (Noncommutative $n$-tori) Let $\theta$ be a real skew-symmetric $n\times n$ matrix. The \emph{noncommutative $n$-torus} $A^n_{\theta}$ is the universal unital C*-algebra generated by unitaries $U_1,\ldots,U_n$ with
$$U_rU_s=\exp(2\pi i\theta_{rs})U_sU_r\,\,\,\text{for all}\,\,\,1\leq r,s\leq n.$$
Moreover, there is a continuous action $\alpha$ of $\mathbb{T}^n$ on $A^n_{\theta}$ by algebra automorphisms, which is on generators given by 
$$\alpha(t).U^{\bf k}:=t.U^{\bf k}:=t^{\bf k}\cdot U^{\bf k}\,\,\,\text{for}\,\,\,{\bf k}\in\mathbb{Z}^n,$$
where $U^{\bf k}:=U^{k_1}_1\cdots U^{k_n}_n$. In particular, $(A^n_{\theta})_{\bf k}=\mathbb{C}\cdot U_{\bf k}$ shows that the triple $(A^n_{\theta},\mathbb{T}^n,\alpha)$ is a trivial noncommutative principal $\mathbb{T}^n$-bundle.

(b) The \emph{smooth noncommutative n-torus} $\mathbb{T}^n_{\theta}$ is the unital subalgebra of smooth vectors for the previous action. Its elements are given by (norm-convergent) sums
\[a=\sum_{{\bf k}\in\mathbb{Z}^n}a_{\bf k}U^{\bf k},\,\,\,\text{with}\,\,\,(a_{\bf k})_{{\bf k}\in\mathbb{Z}^n}\in S(\mathbb{Z}^n).
\]Further, a deeper analysis shows that the induced action of $\mathbb{T}^n$ on $\mathbb{T}^n_{\theta}$ is smooth. Thus, the triple $(\mathbb{T}^n_{\theta},\mathbb{T}^n,\alpha)$ is a smooth trivial noncommutative principal $\mathbb{T}^n$-bundle. 
\end{example}

\noindent
Another hint for the quality of this definition for trivial noncommutative principal torus bundles is the observation that they have a natural counterpart in the theory of Hopf--Galois extensions: In fact, up to a suitable completion, they correspond to the so-called cleft $\mathbb{C}[\mathbb{Z}^n]$-comodule algebras; loosely speaking, being cleft is a triviality condition in the theory of Hopf--Galois extensions:

\begin{remark}
(Realtion to cleft Hopf--Galois extensions). An algebra $A$ is a $\mathbb{C}[\mathbb{Z}^n]$-comodule algebra if and only if $A$ is a $\mathbb{Z}^n$-graded algebra (cf. \cite[Lemma 4.8]{BlMo}). Moreover, we conclude from \cite[Example 2.1.4]{Sch04} that a $\mathbb{Z}^n$-graded algebra $A=\bigoplus_{{\bf k}\in\mathbb{Z}^n}A_{\bf k}$ is a Hopf--Galois extension (of $A_{\bf 0}$) if and only if $A$ is strongly graded, i.e., $A_{\bf k}A_{\bf k'}=A_{\bf k+k'}$ for all ${\bf k},{\bf k'}\in\mathbb{Z}^n$. Now, a short calculation shows that a $\mathbb{C}[\mathbb{Z}^n]$-comodule algebra $A$ is cleft if and only if each grading space $A_{\bf k}$ contains an invertible element. For more background on Hopf--Galois extensions, in particular for the definition of cleft extensions, we refer to \cite[Section 2.2]{Sch04} or \cite[Appendix A]{Wa11a}.
\end{remark}

In view of the previous discussion, it is the next natural step is to work out a convenient localization method for non-commutative algebras or, more generally, for dynamical systems. For this we first note that the idea of localization comes from Algebraic Geometry: Given a point $x$ in some affine variety $X$, one likes to investigate the nature of $X$ in an arbitrarily small neighbourhood of $x$ in the Zariski topology. Now, small neighbourhoods of $x$ in $X$ correspond to large algebraic subsets $Y$. For example, let $Y$ be the zero set of some algebraic function $f$ on $X$, which does not vanish at $x$. Then the affine ring $\mathbb{K}[(X\backslash Y)]$ is obtained from $\mathbb{K}[X]$ by adjoining a multiplicative inverse for $f$, i.e., taking the coproduct of $\mathbb{K}[X]$ with the free ring in one generator and dividing out the ideal $I_f:=\langle 1-ft\rangle$; this is called \emph{inverting} $f$. 

\begin{proposition}\label{K(X)_f=O(X_f)}
\emph{(}cf. \cite[Chapter II, Proposition 2.2 (b)]{Har87}\emph{)} Let $X$ be an affine variety and $f$ an element in the coordinate ring $\mathbb{K}[X]$. If $X_f:=\{x\in X:\,f(x)\neq0\}$, then the map 
\[\mathbb{K}[X]_f\rightarrow \mathcal{O}_X(X_f),\,\,\,\frac{g}{f^n}\mapsto\left(x\mapsto\frac{g(x)}{f^n(x)}\right)
\]is an isomorphism of rings.
\end{proposition}

\noindent
This construction is not valid for the algebra of smooth functions on some manifold $M$. Indeed, if $f\in C^{\infty}(M)$ and $M_f:=\{m\in M:\,f(m)\neq 0\}$, then not every smooth function $g:M_f\rightarrow\mathbb{C}$ is of the form $\frac{h}{f^n}$ for $h\in C^{\infty}(M)$ and some $n\in\mathbb{N}$. This is due to the fact that there are ``too many " smooth functions. Another important remark in this context is that the natural restriction map $r_U:C^{\infty}(M)\rightarrow C^{\infty}(U)$, $f\mapsto f_{\mid U}$ is in general neither injective nor surjective.\\

In Section \ref{SLNS} we present a construction which will fix this problem. To be more precise, we present an appropriate method of localizing (possibly noncommutative) algebras in a smooth way, while in Section 3 we are concerned with the problem of calculating the spectrum of such a localized algebra.\\

Section 4 is devoted to discussing algebra bundles $q:\mathbb{A}\rightarrow M$ with a possibly infinite-dimensional fibre $A$ over a finite-dimensional manifold $M$. In particular, we show how to endow the corresponding space $\Gamma\mathbb{A}$ of sections with a topology that turns it into a (unital) locally convex algebra. The general philosophy here is use the existing results for mapping spaces and reduce the occurring questions of continuity to mapping spaces.\\

In Section 5 we finally prove a smooth analogue of Proposition \ref{K(X)_f=O(X_f)} for the algebra $C^{\infty}(M)$ of smooth functions on some manifold $M$. In fact, we even show that if $A$ is a unital Fr\'echet algebra, $(\mathbb{A},M,A,q)$ an algebra bundle and $f\in C^{\infty}(M,\mathbb{R})$, then the smooth localization of $\Gamma\mathbb{A}$ with respect to $f$ is isomorphic (as a unital Fr\'echet algebra) to $\Gamma\mathbb{A}_{M_f}$. Here, \mbox{$M_f:=\{m\in M:\,f(m)\neq0\}$}. This result means that it is possible to reconstruct $C^{\infty}(U)$ by localization out of data from $C^{\infty}(M)$. The crucial idea is to embed $U$ in an appropriated way as a closed submanifold into $\mathbb{R}\times M$ and, of course, the methods we use come from differential geometry, like the regular value theorem (also known as the submersion theorem). Thus, it is not straightforward if similar results also hold in the context of $C^*$-algebras.\\

Section 6 is dedicated to describing the spectrum of the algebra of sections of an algebra bundle with finite-dimensional commutative fibre. We obtain a beautiful result which connects algebra with geometry:

\begin{theo}
{\it Let $A$ be a finite-dimensional unital commutative algebra. Further, let $(\mathbb{A},M,A,q)$ be an algebra bundle. If $\dim A=n$, then the spectrum $\Gamma_{\Gamma\mathbb{A}}$ is an $n$-fold covering of $M$.}
\end{theo}

In Section \ref{ACNCPT^nB} we use the ideas of Sections \ref{SLNS} to introduce a method of localizing dynamical systems $(A,G,\alpha)$ with respect to elements of $C_A^G$, i.e., with respect to elements of the commutative fixed point algebra of the induced action of $G$ on the center $C_A$ of $A$. In fact, this construction turns out to be the starting point for our approach to noncommutative principal torus bundles: Given a principal bundle $(P,M,G,q,\sigma)$ and an open subset $U$ of $M$ such that $P_U:=q^{-1}(U)$ is trivial, we show that the localization of the induced smooth dynamical system $(C^{\infty}(P),G,\alpha)$ around ``$U$'' leads to the smooth dynamical system $(C^{\infty}(P_U),G,\alpha)$ which is in turn nautrally isomorphic to the smooth dynamical system $(C^{\infty}(U\times G),G,\alpha)$ coming from the trivial principal bundle $(U\times G,U,G,\pr_U,\sigma_G)$.\\

The main goal of Section \ref{a geometric approach to NCPB} is to present a geometrically oriented approach to the noncommutative geometry of principal torus bundles. As already mentioned a trivial noncommutative principal torus bundle is a dynamical system $(A,\mathbb{T}^n,\alpha)$ with the additional property that each isotypic component contains an invertible element. In view of the previous discussion, our main idea is inspired by the classical setting: Loosely speaking, a dynamical system $(A,\mathbb{T}^n,\alpha)$ is called a noncommutative principal $\mathbb{T}^n$-bundle, if it is ``locally" a trivial noncommutative principal $\mathbb{T}^n$-bundle, i.e., Section \ref{ACNCPT^nB} enters the picture. We prove that this approach extends the classical theory of principal torus bundles and a present some noncommutative examples. Indeed, we first show that each trivial noncommutative principal torus bundle carries the structure of a noncommutative principal torus bundle in its own right. We further show that examples are provided by sections of algebra bundles with trivial noncommutative principal torus bundle as fibre, sections of algebra bundles which are pull-backs of principal torus bundles and sections of trivial equivariant algebra bundles. At the end of this section we present a very concrete example.\\

Another advantage of our approach is that it seems to have a nice classification theory. Indeed, in \cite[Section 4]{Wa11b} we gave a complete classification of trivial noncommutative principal torus bundles (up to a suitable completion). Thus it is, of course, a natural ambition to work out a classification theory for (non-trivial) noncommutative principal torus bundles. The final section is therefore devoted to hints and ideas for an appropriate classification theory for (non-trivial) noncommutative principal torus bundles.\\ 

Appendix A is devoted to some results concerning the projective tensor product of locally convex spaces. In Appendix B we discuss some useful results on the smooth exponential law which will be used several times within this paper.

\section*{Preliminaries and Notations} All manifolds appearing in this paper are assumed to be finite-dimensional, paracompact, second countable and smooth. For the necessary background on (principal) bundles and vector bundles we refer to \cite{KoNo63}. All algebras are assumed to be complex if not mentioned otherwise. Given an algebra $A$, we write $\Gamma_A:=\Hom_{\text{alg}}(A,\mathbb{C})\backslash\{0\}$ (with the topology of pointwise convergence on $A$) for the spectrum of $A$ and $\Aut(A)$ for the corresponding group of automorphisms in the category of $A$. Moreover, a dynamical system is a triple $(A,G,\alpha)$, consisting of a unital locally convex algebra $A$, a topological group $G$ and a group homomorphism $\alpha:G\rightarrow\Aut(A)$, which induces a continuous action of $G$ on $A$.

\section*{Acknowledgements} We thank K.-H. Neeb and Helge Gl\"ockner for fruitful and inspiring discussions. Moreover, we thank Christoph Zellner for proofreading of parts of this paper.

\section{Smooth Localization of Noncommutative Spaces}\label{SLNS}

From the viewpoint of Noncommutative Differential Geometry, one might ask whether Proposition \ref{K(X)_f=O(X_f)} is still true for the algebra of smooth function on a manifold $M$. Unfortunately, if $f$ is an element in $C^{\infty}(M)$ and $M_f:=\{m\in M:\,f(m)\neq 0\}$, then not every smooth function $g:M_f\rightarrow\mathbb{K}$ is of the form $\frac{h}{f^n}$ for $h\in C^{\infty}(M)$ and some $n\in\mathbb{N}$. In this section we provide a construction which will fix this problem. For the following propositions we recall the smooth compact open topology for smooth vector-valued function spaces of Definition \ref{smooth compact open topology} and the projective tensor product topology of Definition \ref{proj. tensor product III}.

\begin{proposition}\label{C(M,A) loc. con. algebra}
If $M$ is a manifold and $A$ a locally convex algebra, then $C^{\infty}(M,A)$ is a locally convex algebra.
\end{proposition}

\begin{proof} 
To prove the claim we just have to verify that the multiplication in $C^{\infty}(M,A)$ is continuous:

(i) For this we first note that the tangent space $TA$ of $A$ carries a natural locally convex algebra structure given by the tangent functor $T$, i.e., defined by 
\[(a,v)(a',v'):=(aa',av'+va').
\]If $m_A:A\times A\rightarrow A$ is the multiplication of $A$, then $T(m_A):TA\times TA\cong T(A\times A)\rightarrow TA$ is the multiplication of $TA$. Iterating this process, we obtain a locally convex algebra structure on $T^nA$ for each $n\in\mathbb{N}$.

(ii) For elements $f,g\in C^{\infty}(M,A)$ the functoriality of $T$ implies that
\[T^n(fg)=T^n(m_A\circ(f,g))=T^n(m_A)\circ T^n(f,g)=T^n(m_A)\circ(T^n f,T^n g)=T^n f\cdot T^n g.
\]In particular, we conclude that the embedding
\[C^{\infty}(M,A)\hookrightarrow\prod_{n\in\mathbb{N}_0}C(T^nM,T^nA),\,\,\,f\mapsto(T^nf)_{n\in\mathbb{N}_0},
\]is a morphism of algebras. By \cite[Proposition D.5.7]{Wa11a}, the multiplication on the product algebra on the right is continuous, which implies the continuity of the multiplication on $C^{\infty}(M,A)$.
\end{proof}

\begin{remark}\label{proj. tensor product IV}
If $M$ is a manifold and $E$ a complete locally convex space, then the locally convex space $C^{\infty}(M,E)$ is complete (cf. \cite[Proposition 4.2.15 (a))]{Gl05}. Moreover, a short observation shows that the subspace generated by the image of the bilinear map $p:C^{\infty}(M)\times E\rightarrow C^{\infty}(M,E),\,\,\,(f,e)\mapsto f\cdot e$ consists of functions whose image is contained in a finite dimensional subspace of $E$. Writing $C_{\text{\emph{fin}}}^{\infty}(M,E)$ for this subspace, we conclude from an example in [Gro55], Chapter 2, \S 3.3, Theorem 13 that the map
\[\phi:C^{\infty}(M)\otimes E\rightarrow C_{\text{\emph{fin}}}^{\infty}(M,E),\,\,\,f\otimes e\mapsto f\cdot e
\]is an isomorphism of locally convex spaces and can therefore be extended to an isomorphism $\Phi:C^{\infty}(M)\widehat{\otimes}E\rightarrow C^{\infty}(M,E)$ of complete locally convex spaces. The crucial point here is to use the fact that $C^{\infty}(M)$ is a nuclear space. Note that this isomorphism means that the smooth compact open topology and the projective tensor product topology coincides on $C^{\infty}(M)\otimes E$.
\end{remark}

\begin{proposition}\label{proj. tensor product for algebras}
If $M$ is a manifold and $A$ is a complete locally convex algebra, then there is a unique locally convex algebra structure on the locally convex space $C^{\infty}(M)\widehat{\otimes}A$ for which $C^{\infty}(M)\widehat{\otimes}A$ and $C^{\infty}(M,A)$ are isomorphic as complete locally convex algebras.
\end{proposition}

\begin{proof}
According to Proposition \ref{AoB as l.c. algebra}, $C^{\infty}(M)\otimes A$ is a locally convex algebra. Thus, \cite[Corollary D.1.7]{Wa11a} applied to the multiplication map of $C^{\infty}(M)\otimes A$ implies that $C^{\infty}(M)\widehat{\otimes}A$ is also a locally convex algebra. If $\widehat{m}$ denotes the multiplication map of $C^{\infty}(M)\widehat{\otimes}A$ and $m$ the multiplication map of the locally convex algebra $C^{\infty}(M,A)$ (cf. Proposition \ref{C(M,A) loc. con. algebra}), then it remains to verify the identity
\begin{align}
\Phi\circ\widehat{m}=m\circ(\Phi\times\Phi),\label{verify identity}
\end{align}
where $\Phi:C^{\infty}(M)\widehat{\otimes}A\rightarrow C^{\infty}(M,A)$ denotes the isomorphism of locally convex spaces stated in Remark \ref{proj. tensor product IV}: In fact, an easy observation shows that (\ref{verify identity}) holds on $C^{\infty}(M)\otimes A$ and thus the claim follows from the principle of extension of identities.
\end{proof}

The next definition is crucial for the aim of this paper since it is the beginning of a smooth localization method:

\begin{definition}\label{sm. loc.}
(Smooth localization of algebras). For $n\in\mathbb{N}$ and a unital locally convex algebra $A$ we write
\[A\{t_1,\ldots,t_n\}:=C^{\infty}(\mathbb{R}^n,A)
\]for the unital locally convex algebra of smooth $A$-valued functions on $\mathbb{R}^n$ (cf. Proposition \ref{C(M,A) loc. con. algebra}). We further define for each $1\leq i\leq n$ and $a\in A$ a smooth $A$-valued function on $\mathbb{R}^n$ by 
\[f^i_a:\mathbb{R}^n\rightarrow A,\,\,\,(t_1,\ldots,t_n)\mapsto 1_A-t_ia.
\]If $a_1,\ldots,a_n\in A$, then we write $I_{a_1,\ldots,a_n}:=\overline{\langle f^1_{a_1},\ldots,f^n_{a_n}\rangle}$ for the closure of the two-sided ideal generated by these functions. Finally, we write
\[A_{\{a_1,\ldots,a_n\}}:=A\{t_1,\ldots,t_n\}/I_{a_1,\ldots,a_n}
\]for the corresponding locally convex quotient algebra and
\[\pi_{\{a_1,\ldots,a_n\}}:A\{t_1,\ldots,t_n\}\rightarrow A_{\{a_1,\ldots,a_n\}}
\]for the corresponding continuous quotient homomorphism. The algebra $A_{\{a_1,\ldots,a_n\}}$ is called the (\emph{smooth}) \emph{localization} of $A$ with respect to $a_1,\ldots,a_n$.
\end{definition}

We continue with a bunch of remarks on the smooth localization of an algebra:

\begin{remark}
For $n\in\mathbb{N}$ and a complete locally convex algebra $A$, Proposition \ref{proj. tensor product for algebras}, applied to $M=\mathbb{R}^n$, implies that
\[C^{\infty}(\mathbb{R}^n)\widehat{\otimes}A\cong C^{\infty}(\mathbb{R}^n,A)
\]as complete locally convex algebras. This result corresponds to the classical picture in commutative algebra of adjoining $n$ indeterminate elements to a ring $R$ by taking the coproduct of $R$ and the free $\mathbb{K}$-algebra in $n$ generators: 
\[\mathbb{K}[t_1,\ldots,t_n]\otimes R\cong R[t_1,\ldots,t_n].
\]In particular, the algebra $C^{\infty}(\mathbb{R}^n,A)$ may be thought of the outcome of adjoining $n$ indeterminates to the algebra $A$ in a smooth way.
\end{remark}

\begin{remark}\label{rem to sm. loc.}
Unlike in ordinary commutative algebra, where every element of the ``classical localization'' $A_a$ of $A$ with respect to $a\in A$ can be written in the form $\frac{x}{a^n}$ for some $x\in A$ and $n\in\mathbb{N}$, we do, not even in the case $n=1$, have an explicit description of the elements of $A_{\{a_1,\ldots,a_n\}}$. Moreover, the outcome is, in general, quite different from the classical localization procedure; for example, if one localizes an algebra $A$ with respect to a non-zero quasi-nilpotent element $a$, i.e., an element with $\spec(a)=\{0\}$ which is not nilpotent, then $A_{\{a\}}=0$ since the function $f_a$ is invertible in $A\{t\}$. On the other hand we have $0\neq\frac{a}{1}\in A_a$. We will point out another difference in Corollary \ref{C(M)_f=C(M_f)}. Nevertheless, by the universal property of the classical localization procedure, we have a canonical homomorphism from $A_a$ to $A_{\{a\}}$.
\end{remark}

\begin{remark}\label{inverting 1 and 0} 
Let $A$ be an arbitrary (possibly noncommutative) unital algebra. Then localizing $A$ with respect to 0 leads to the zero-algebra, i.e.,
$A_0={\bf0}$ and $A_{\{0\}}={\bf0}$. On the other hand, if $a$ in $A$ is an invertible element, then localization of $A$ with respect to $a$ changes nothing, i.e., there is a canonical isomorphism between $A$ and $A_a$. This is not clear at all for the smooth localization $A_{\{a\}}$.
\end{remark}

\begin{remark}\label{A_z frechet again}
(Fr\'echet algebras). If $A$ is a unital Fr\'{e}chet algebra, then the same holds for $A_{\{a_1,\ldots,a_n\}}$. Indeed, this assertion follows from the fact that the quotient of a Fr\'{e}chet space by a closed subspace is again Fr\'{e}chet (cf. \cite[\S 3.5]{Bou55}). 
\end{remark}

\begin{proposition}\label{inverting 1}
Let $A$ be a complete unital locally convex algebra. Then the map
\[\ev_{1}:A\{t\}\rightarrow A,\,\,\,f\mapsto f(1)
\]is a surjective morphism of \emph{(}complete\emph{)} unital locally convex algebras with kernel 
\[I_{1_A}:=(t-1)\cdot A\{t\}.
\]
\end{proposition}

\begin{proof}
(i) According to \cite[Proposition I.2]{NeWa07}, the evaluation map
\[\ev_{\mathbb{R}}:C^{\infty}(\mathbb{R},A)\times\mathbb{R}\rightarrow A,\,\,\,(f,r)\mapsto f(r)
\]is smooth. In particular, the map $\ev_{1}$ is continuous. Further, a short observation shows that $\ev_{1}$ is surjective and a homomorphism of algebras. Therefore, $\ev_{1}$ is a surjective morphism of (complete) unital locally convex algebras.

(ii) It remains to determine the kernel $I:=\ker\ev_{1}$: Clearly, $I_{1_A}\subseteq I$. Therefore, let $f\in A\{t\}$ with $f(1)=0$. We define a smooth $A$-valued function on $\mathbb{R}$ by
\[g:\mathbb{R}\rightarrow A,\,\,\,g(t):=\int^1_0f'(s(t-1)+1)\,ds.
\]Now, an easy calculation leads to
\[f(t)=(t-1)\cdot g(t),
\]i.e., to $f\in I_{1_A}$ and thus $I_{1_A}=I$.
\end{proof}

\begin{corollary}\label{inverting 1 iso}
In the situation of Proposition \ref{inverting 1}, the map
\[\varphi:A_{\{1_A\}}\rightarrow A,\,\,\,f+I_{1_A}\mapsto f(1)
\]is an isomorphism of \emph{(}complete\emph{)} locally convex algebras.
\end{corollary}

\begin{proof}
In view of Proposition \ref{inverting 1} and the definition of the quotient topology, the map $\varphi$ is a bijective morphism of locally convex algebras. Further, a short observation shows that the map 
\[\psi:=\pi_{\{1_A\}}\circ i:A\rightarrow A_{\{1_A\}},
\]where $i:A\rightarrow A\{t\}$ denotes the canonical inclusion, is a continuous inverse of the map $\varphi$.
\end{proof}



In the forthcoming chapter we will need the following property of $A_{\{a\}}$:

\begin{lemma}\label{A_{a...} is a (left) A-modul}
For all $a_1,\ldots,a_n\in A$ there is a continuous $A$-bimodule structure on $A_{\{a_1,\ldots,a_n\}}$, given for all $a,a'\in A$ and $f\in A\{t_1,\ldots,t_n\}$ by 
\[a.[f].a':=[afa']=afa'+I_{a_1,\ldots,a_n}.
\]
\end{lemma}

\begin{proof}
The claim follows from the continuity of the $A$-bimodule structure on $A\{t_1,\ldots,t_n\}$, given for all $a,a'\in A$ and $f\in A\{t_1,\ldots,t_n\}$ by 
\[(a.f.a')(t_1,\ldots,t_n):=af(t_1,\ldots,t_n)a',
\]and the definition of the quotient topology.
\end{proof}

\begin{remark}
(Another localization method). Another interesting localization method can be found in \cite[Chapter 3]{NaSa03}. Indeed, given a commutative $\mathbb{R}$-algebra $A$ and an open subset $U$ of the \emph{real spectrum} $\Hom_{\text{alg}}(A,\mathbb{R})\backslash\{0\}$ of $A$ (with the topology of pointwise convergence on $A$),
the authors of \cite{NaSa03} define $A_U$ to be the ring of fractions of $A$ with respect to the multiplicative subset of all elements in $A$ without zeroes in $U$, i.e., elements in $A_U$ are (equivalence classes of) fractions $\frac{a}{s}$, where $a,s\in A$ and $s(\chi):=\chi(s)\neq 0$ for all $\chi\in U$. In particular, they show that if $A=C^{\infty}(\mathbb{R}^n,\mathbb{R})$, then $A_U=C^{\infty}(U,\mathbb{R})$ for each open subset $U$ of $\mathbb{R}^n$ (cf. \cite[Chapter 3, Example 3.3]{NaSa03}). Of course, an important restriction of this approach is the commutativity of the algebra $A$.
\end{remark}


\section{The Spectrum of $A_{\{a_1,\ldots,a_n\}}$}\label{the spectrum of A_a section}

Before coming up with concrete examples of smoothly localized algebras, we will first be concerned with the problem of calculating the spectrum of $A_{\{a_1,\ldots,a_n\}}$ for some unital locally convex algebra $A$ and elements $a_1,\ldots,a_n\in A$. In the following we write $\Gamma^{\text{cont}}_A$ for the set of continuous characters of $A$. We start with discussing the spectrum of the algebra of smooth functions on a manifold:

\begin{lemma}\label{spec of C(M,K) set}
If $M$ is a manifold, then each character $\chi:C^{\infty}(M)\rightarrow\mathbb{C}$ is an evaluation in some point $m\in M$.
\end{lemma}

\begin{proof}
A proof of this statement can be found in \cite[Corollary 4.3.2]{Wa11a}.
\end{proof}

The next proposition shows that the correspondence between $M$ and $\Gamma_{C^{\infty}(M)}$ is actually a topological isomorphism:

\begin{proposition}\label{spec of C(M) top}
Let $M$ be a manifold. Then the map
\begin{align}
\Phi_M:M\rightarrow \Gamma_{C^{\infty}(M)},\,\,\,m\mapsto \delta_m\notag.
\end{align}
is a homeomorphism.
\end{proposition}

\begin{proof}
(i) The surjectivity of $\Phi$ follows from Lemma \ref{spec of C(M,K) set}.
To show that $\Phi$ is injective, choose elements $m\neq m'$ of $M$. Since $M$ is manifold, there exists a function $f$ in $C^{\infty}(M)$ with $f(m)\neq f(m')$. Then 
\[\delta_m(f)=f(m)\neq f(m')=\delta_{m'}(f)
\]implies that $\delta_m\neq\delta_{m'}$, i.e., $\Phi$ is injective. 

(ii) Next, we show that $\Phi$ is continuous: Let $m_n\rightarrow m$ be a convergent sequence in $M$. Then we have 
\[\delta_{m_n}(f)=f(m_n)\rightarrow f(m)=\delta_m(f)\,\,\,\text{for all}\,\,\,f\,\,\,\text{in}\,\,\,C^{\infty}(M),
\]i.e., $\delta_{m_n}\rightarrow\delta_m$ in the topology of pointwise convergence. Hence, $\Phi$ is continuous. 

(iii) We complete the proof by showing that $\Phi$ is an open map: For this let $U$ be an open subset of $M$, $m_0$ in $U$ and $h$ a smooth real-valued function with $h(m_0)\neq 0$ and $\supp(h)\subset U$. Since the map
\[\delta_h:\Gamma_{C^{\infty}(M)}\rightarrow\mathbb{C},\,\,\,\delta_m\mapsto h(m)
\]is continuous, a short calculations shows that $\Phi(U)$ is a neighbourhood of $m_0$ containing the open subset $\delta_h^{-1}(\mathbb{C}^{\times})$. Hence, $\Phi$ is open.
\end{proof}

The following observation is well-known, but by a lack of a reference, we give the proof:

\begin{lemma}\label{spectrum of tensor products}
Let $A$ and $B$ be two unital locally convex algebras. Then the map
\[\Phi:\Gamma^{\emph{\text{cont}}}_A\times\Gamma^{\emph{\text{cont}}}_B\rightarrow\Gamma^{\emph{\text{cont}}}_{A\otimes B},\,\,\,(\chi_A,\chi_B)\rightarrow \chi_A\otimes\chi_B
\]is a homeomorphism.
\end{lemma}

\begin{proof}
(i) We first note that Lemma \ref{continuity of maps EoF to E'oF'} implies that the map $\Phi$ is well-defined. To show that the map $\Phi$ is bijective, we recall that each simple tensor $a\otimes b$ can be written as $(a\otimes 1_B)\cdot(1_A\otimes b)$. Therefore, every continuous character $\chi:A\otimes B\rightarrow\mathbb{C}$ is uniquely determined by the elements of the form $(a\otimes 1_B)$ and $(1_A\otimes b)$ for $a\in A$ and $b\in B$. In particular, the restriction of $\chi$ to the subalgebra $A\otimes 1_B$, resp., $1_A\otimes B$ corresponds to a continuous character of $A$, resp., $B$, i.e.,
\[\chi=\chi_A\otimes\chi_B\,\,\,\text{for}\,\,\,\chi_A\in\Gamma_A\,\,\,\text{and}\,\,\,\chi_B\in\Gamma_B.
\]Hence, $\Phi$ is surjective. A similar argument shows that the map $\Phi$ is injective.

(ii) Finally, we leave it as an easy exercise to the reader to verify the continuity of $\Phi$ and its inverse $\Phi^{-1}$.
\end{proof}

\begin{remark}\label{CIA stuff}
(Continuous inverse algebras). (a) A unital locally convex algebra $A$ is called continuous inverse algebra, or CIA for short, if its group of units $A^{\times}$ is an open subset in $A$ and the inversion map $\iota:A^{\times}\rightarrow A^{\times},\,\,\,a\mapsto a^{-1}$ is continuous at $1_A$. They are encountered in K-theory and noncommutative geometry, usually as dense unital subalgebras of $C^*$-algebras. 

(b) For a compact manifold $M$, the Fr\'echet algebra of smooth functions $C^{\infty}(M)$ is the prototype of such a continuous inverse algebra. More generally, if $M$ is a compact manifold and $A$ a continuous inverse algebra, then $C^{\infty}(M,A)$, equipped with the smooth compact open topology, is a continuous inverse algebra (cf. \cite[Proposition 7.1]{Gl02}). The example $C^{\infty}(\mathbb{R},\mathbb{R})$ shows that, if $A$ is a continuous inverse algebra, $C^{\infty}(M,A)$ need not have an open unit group.

(c) If $A$ is a commutative CIA, then each character $\chi:A\rightarrow\mathbb{C}$ is continuous. Moreover, given a maximal proper ideal $I$ in $A$, then $I$ is the kernel of some character $\chi:A\rightarrow\mathbb{C}$. A reference for the last two statements is \cite[Chapter 2]{Bi04}.
\end{remark}

We are now ready to prove the following theorem on continuous characters of smooth vector-valued function spaces:

\begin{theorem}\label{spectrum of  C(M,A)}
Let $M$ be a manifold and $A$ be a unital locally convex algebra. Further let $B:=C^{\infty}(M,A)$. Then the following assertions hold:
\begin{itemize}
\item[(a)]
If $M$ is compact and $A$ is a CIA, then each maximal ideal of $B$ is closed.
\item[(b)]
If $A$ is complete, then each continuous character $\varphi:B\rightarrow\mathbb{C}$ is an evaluation homomorphism 
\[\chi\circ\delta_m:B\rightarrow\mathbb{C},\,\,\,f\mapsto \chi(f(m))
\]for some $m\in M$ and $\chi\in\Gamma^{\emph{\text{cont}}}_A$.
\end{itemize}
\end{theorem}

\begin{proof}
(a) We first note that Remark \ref{CIA stuff} (ii) implies that $B$ is a CIA. Hence, the unit group $B^{\times}=C^{\infty}(M,A^{\times})$ is an open subset of $B$. If $I\subseteq A$ is a maximal ideal, then $I$ intersects $B^{\times}$ trivially, and since $B^{\times}$ is open, the same holds for the closure $\bar{I}$. Hence, $\bar{I}$ also is a proper ideal, so that the maximality of $I$ implies that $I$ is closed.

(b) According to Proposition \ref{proj. tensor product for algebras}, we know that $C^{\infty}(M)\widehat{\otimes}A$ and $C^{\infty}(M,A)$ are isomorphic as complete unital locally convex algebras. In particular, $C^{\infty}(M)\otimes A$ is a dense unital subalgebra of $C^{\infty}(M,A)$, and thus each (continuous) character $\varphi:B\rightarrow\mathbb{C}$ restricts to a character on $C^{\infty}(M)\otimes A$, which is, by Lemma \ref{spectrum of tensor products}, an evaluation in some point $(m,\chi)\in M\times\Gamma^{\text{cont}}_A$. This character is continuous with respect to the projective tensor product topology and can therefore be uniquely extended to a continuous character on $B$.
\end{proof}

\begin{corollary}\label{spectrum of C(M,A) A CIA} 
Let $M$ be a manifold and $A$ be a complete CIA. If $B:=C^{\infty}(M,A)$, then each continuous character $\varphi:B\rightarrow\mathbb{K}$ is an evaluation homomorphism 
\[\chi\circ\delta_m:B\rightarrow\mathbb{C},\,\,\,f\mapsto \chi(f(m))
\]for some $m\in M$ and $\chi\in\Gamma_A$. In particular, if $M$ is compact, then each character of $B$ is an evaluation homomorphism.
\end{corollary}

\begin{proof}
The first part of the Corollary immediately follows from Theorem \ref{spectrum of  C(M,A)} (b) and Remark \ref{CIA stuff} (c), which states that every character of $A$ is continuous. If $M$ is additionally compact, then Remark \ref{CIA stuff} (b) implies that $B$ is a (complete) CIA. Therefore, the second assertion (again) follows from Remark \ref{CIA stuff} (c).
\end{proof}

\begin{remark}
It would be nice to find a purely algebraic proof of Corollary \ref{spectrum of C(M,A) A CIA}, which shows that, even in the non-compact case, every (!) character of $B$ is an evaluation homomorphism. By Lemma \ref{spec of C(M,K) set} this is, for example, true for $A\in\{\mathbb{R},\mathbb{C}\}$ and should also work for function algebras on compact spaces, i.e., for commutative C*-algebras. For further investigations on this question, we refer to the paper \cite{Ou07}.
\end{remark}

\begin{corollary}
If $M$ is compact and $A$ a complete commutative CIA, then each maximal proper ideal of $B:=C^{\infty}(M,A)$ is the kernel of an evaluation homomorphism
\[\chi\circ\delta_m:B\rightarrow\mathbb{C},\,\,\,f\mapsto \chi(f(m))
\]for some $m\in M$ and $\chi\in\Gamma_A$.
\end{corollary}

\begin{proof}
According to Remark \ref{CIA stuff} (b), $B$ is a commutative CIA. Hence, Remark \ref{CIA stuff} (c) implies that each maximal proper ideal of $B$ is the kernel of a character of $B$. Therefore, the claim follows from Corollary \ref{spectrum of C(M,A) A CIA}.
\end{proof}

The following theorem shows that, under a certain condition on the continuous spectrum of $A$, the bijection of Theorem \ref{spectrum of C(M,A)} (b) becomes an isomorphism of topological spaces:

\begin{theorem}\label{spec of C(M,A) top}
Let $M$ be a manifold and $A$ be a complete unital locally convex algebra. Further, let $B:=C^{\infty}(M,A)$ and assume that $\Gamma^{\emph{\text{cont}}}_A$ is locally equicontinuous. Then the map
\[\Phi:M\times\Gamma^{\emph{\text{cont}}}_A\rightarrow\Gamma^{\emph{\text{cont}}}_B,\,\,\,(m,\chi)\mapsto\chi\circ\delta_m
\]is a homeomorphism.
\end{theorem}

\begin{proof}
(i) Since $A$ is unital, $C^{\infty}(M)$ is a central subalgebra of $B$. Moreover, we note that $A$ is embedded in $B$ as the constant-valued functions. Hence, $\Phi(m,\chi)=\Phi(m',\chi')$ implies that $m=m'$ and $\chi=\chi'$, i.e., $\Phi$ is injective. Further, Theorem \ref{spectrum of C(M,A)} (b) implies that the map $\Phi$ is surjective. 

(ii) Next, we prove the continuity of the inverse map $\Phi^{-1}$: For this we choose a convergent net $\varphi_i=\chi_i\circ\delta_{m_i}\rightarrow\chi\circ\delta_m=\varphi$ in $\Gamma^{\text{cont}}_B$. Because $C^{\infty}(M)$ is a central subalgebra of $B$, we conclude that $\delta_{m_i}\rightarrow \delta_m$ in $\Gamma_{C^{\infty}(M)}$. By Proposition \ref{spec of C(M) top}, the last condition is equivalent to $m_i\rightarrow m$ in $M$. Since $A$ is embedded in $B$ as the constant-valued functions, we get $\chi_i\rightarrow\chi$ in $A$.

(iii) To prove continuity of $\Phi$, let $(m_0,\chi_0)\in M\times\Gamma^{\text{cont}}_A$, $\epsilon>0$ and $f\in B$. We first choose an equicontinuous neighbourhood $V$ of $\chi_0$ in $\Gamma^{\text{cont}}_A$ such that
\[V\subseteq\left\{\chi\in\Gamma^{\text{cont}}_A:\,\vert(\chi-\chi_0)(f(m_0))\vert<\frac{\epsilon}{2}\right\}.
\]Next, we choose a neighbourhood $W$ of $f(m_0)$ in $A$ such that 
\[\vert\chi(a-f(m_0))\vert<\frac{\epsilon}{2}
\]for all $a\in W$ and $\chi\in V$. Finally, we choose a neighbourhood $U$ of $m_0$ in $M$ such that $f(m)\in W$ for all $m\in U$. Then $(m,\chi)\in U\times V$ implies $f(m)\in W$ and $\chi\in V$ and therefore
\[\vert \chi(f(m)-f(m_0))\vert<\frac{\epsilon}{2}\,\,\,\text{and}\,\,\,\vert(\chi-\chi_0)(f(m_0))\vert<\frac{\epsilon}{2}.
\]It follows that
\begin{align}
&\vert\Phi(m,\chi)(f)-\Phi(m_0,\chi_0)(f)\vert=\vert\chi(f(m))-\chi_0(f(m_0))\vert\notag\\
&\leq\vert(\chi-\chi_0)(f(m_0))\vert+\vert \chi(f(m)-f(m_0))\vert<\epsilon\notag
\end{align}
for all $(m,\chi)\in U\times V$.
\end{proof}

\begin{remark}\label{equicont}
(Sources of algebras with equicontinuous spectrum). (a) The spectrum $\Gamma_A$ of each CIA $A$ is equicontinuous. In fact, let $U$ be a balanced $0$-neighbourhood such that $U\subseteq 1_A-A^{\times}$. Then $\vert\Gamma_A(U)\vert<1$ (cf. the proof of \cite[Lemma 2.2.4]{Bi04}).

(b) Moreover, if $A$ is a \emph{$\rho$-seminormed} algebra, then \cite[Corollary 7.3.9]{Ba00} implies that $\Gamma^{\text{cont}}_A$ is equicontinuous.
\end{remark}

\begin{definition}\label{D(a)}
For a unital algebra $A$ and $a\in A$ we write $D(a):=\{\chi\in\Gamma_A:\,\chi(a)\neq 0\}$ for the set of characters which do not vanish on $a$.
Moreover, if $a_1,\ldots,a_n\in A$, then we define 
\[D(a_1,\ldots,a_n):=\bigcap^n_{i=1}D(a_i).
\]
\end{definition}

\begin{lemma}\label{D(a) is open}
Each set of the form $D(a_1,\ldots,a_n)$ is an open subset of $\Gamma_A$.
\end{lemma}

\begin{proof}
By the definition of $D(a_1,\ldots,a_n)$, it suffices to show that each set of the form $D(a)$ is open. Therefore, let $a\in A$ be arbitrary and note that the function $\ev_a:\Gamma_A\rightarrow\mathbb{C}$, $\chi\mapsto\chi(a)$ is continuous. The claim now follows from $D(a)=(\ev_a)^{-1}(\mathbb{C}^{\times})$.
\end{proof}

Next we want to describe the spectrum of $A_{\{a\}}$ for a complete locally convex algebra $A$ and an element $a\in A$. We first need the following proposition:

\begin{proposition}\label{spectrum of quotient algebras}
Let $A$ be a unital locally convex algebra and let $I$ be a closed two-sided ideal of $A$. Further, let $\pi:A\rightarrow A/I$ denote the quotient homomorphism. Then the map
\[\Phi:\left\{\chi\in\Gamma^{\emph{cont}}_A:\,\chi(I)=\{0\}\right\}\rightarrow \Gamma^{\emph{cont}}_{A/I},\,\,\,\chi\mapsto \left(a+I\mapsto \chi(a)\right)
\]is a homeomorphism with inverse given by
\[\Psi:\Gamma^{\emph{cont}}_{A/I}\rightarrow\left\{\chi\in\Gamma^{\emph{cont}}_A:\,\chi(I)=\{0\}\right\},\,\,\,\chi\mapsto \chi\circ\pi.
\]
\end{proposition}

\begin{proof}
Since the quotient homomorphism $\pi:A\rightarrow A/I$ becomes continuous with respect to the quotient topology, the map $\Phi$ is well-defined. Now, a short calculation shows that both maps $\Phi$ and $\Psi$ are continuous and inverse to each other. 
\end{proof}

\begin{theorem}\label{spec A_a}
Let $A$ be a complete unital locally convex algebra such that $\Gamma^{\emph{\text{cont}}}_A$ is locally equicontinuous. If $a_1,\ldots,a_n\in A$, then the map
\[\Phi_{\{a_1,\ldots,a_n\}}:D(a_1,\ldots,a_n)\rightarrow\Gamma^{\emph{cont}}_{A_{\{a_1,\ldots,a_n\}}},\,\,\,\Phi_{\{a_1,\ldots,a_n\}}(\chi)([f]):=\chi\left(f\left(\frac{1}{\chi(a_1)},\ldots,\frac{1}{\chi(a_n)}\right)\right)
\]is a homeomorphism.
\end{theorem}

\begin{proof}
(i) A short observation shows that the map
\[\phi_{\{a_1,\ldots,a_n\}}:D(a_1,\ldots,a_n)\rightarrow\mathbb{R}^n\times\Gamma^{\text{cont}}_A,\,\,\,\chi\mapsto\left(\frac{1}{\chi(a_1)},\ldots,\frac{1}{\chi(a_n)},\chi\right)
\]is clearly a homeomorphism onto its image. 

(ii) To proceed we show that the set $X:=\im(\phi_{\{a_1,\ldots,a_n\}})$ is homeomorphic to $\Gamma^{\text{cont}}_{A_{\{a_1,\ldots,a_n\}}}$: For this we first put $B:=A\{t_1,\ldots,t_n\}=C^{\infty}(\mathbb{R}^n,A)$. Then Proposition \ref{spectrum of quotient algebras} implies that the space $\Gamma^{\text{cont}}_{A_{\{a_1,\ldots,a_n\}}}$ is homeomorphic to $\{\chi\in\Gamma^{\text{cont}}_B:\,\chi(I_{a_1,\ldots,a_n})=\{0\}\}$. Furthermore, Theorem \ref{spec of C(M,A) top} implies that this last set is homeomorphic to
\[\left\{(r_1,\ldots,r_n,\chi)\in\mathbb{R}^n\times\Gamma^{\text{cont}}_A:\chi\circ\delta_{(r_1,\ldots,r_n)}(I_{a_1,\ldots,a_n})=\{0\}\right\}.
\]Next, we observe that the condition $\chi\circ\delta_{(r_1,\ldots,r_n)}(I_{a_1,\ldots,a_n})=\{0\}$ is equivalent to $r_i\chi(a_i)=1$ for all $1\leq i\leq n$, i.e., we have
\[\left\{(r_1,\ldots,r_n,\chi)\in\mathbb{R}^n\times\Gamma^{\text{cont}}_A:\delta_{(r_1,\ldots,r_n,\chi)}(I_{a_1,\ldots,a_n})=\{0\}\right\}=X.
\]The corresponding homeomorphism $\varphi_{\{a_1,\ldots,a_n\}}:X\rightarrow\Gamma^{\text{cont}}_{A_{\{a_1,\ldots,a_n\}}}$ is given by
\[\varphi_{\{a_1,\ldots,a_n\}}\left(\left(\frac{1}{\chi(a_1)},\ldots,\frac{1}{\chi(a_n)},\chi\right)\right)([f]):=\chi\left(f\left(\frac{1}{\chi(a_1)},\ldots,\frac{1}{\chi(a_n)}\right)\right).
\]

(iii) The claim now follows from $\Phi_{\{a_1,\ldots,a_n\}}=\varphi_{\{a_1,\ldots,a_n\}}\circ\phi_{\{a_1,\ldots,a_n\}}$.
\end{proof}

\section{The Space of Sections of an Algebra Bundle as a Locally Convex Algebra}\label{SSABLCA}

In this section we are dealing with algebra bundles $q:\mathbb{A}\rightarrow M$ with a possibly infinite-dimensional fibre $A$ over a finite-dimensional manifold $M$ and show how to endow the corresponding space $\Gamma\mathbb{A}$ of sections with a topology that turns it into a (unital) locally convex algebra.

\begin{definition}\label{algebra bundle}
(Algebra bundles). Let $A$ be a locally convex algebra. An \emph{algebra bundle} (with fibre $A$) is a quadruple $(\mathbb{A},M,A,q)$, consisting of an infinite-dimensional manifold $\mathbb{A}$, a manifold $M$ and a smooth map $q:\mathbb{A}\rightarrow M$, with the following property: All fibres $\mathbb{A}_m$, $m\in M$, carry algebra structures, and each point $m\in M$ has an open neighbourhood $U$ for which there exists a diffeomorphism
\[\varphi_U:U\times A\rightarrow q^{-1}(U)=\mathbb{A}_U,
\]satisfying $q\circ\varphi_U=p_U$ and all maps
\[\varphi_{U,x}:A\rightarrow\mathbb{A}_x,\,\,\,a\mapsto\varphi_U(x,a)
\]are algebra isomorphisms.
\end{definition}

\begin{definition}\label{sections of algebra bundles}
(Sections of algebra bundles). Let $(\mathbb{A},M,A,q)$ be an algebra bundle. Then the corresponding space
\[\Gamma\mathbb{A}:=\{s\in C^{\infty}(M,\mathbb{A}):q\circ s=\id_M\}
\]of \emph{smooth sections} carries the structure of an algebra. Indeed, given two sections $s_1,s_2\in\Gamma\mathbb{A}$ and $\lambda\in\mathbb{K}$, the vector space structure on $\Gamma\mathbb{A}$ is defined by
\[(s_1+s_2)(m):=s_1(m)+s_2(m)\,\,\,\text{and}\,\,\,(\lambda s)(m):=\lambda s(m).
\]Moreover, the product $(s_1\cdot s_2)(m):=s_1(m)\cdot s_2(m)$ defines a multiplication map on $\Gamma\mathbb{A}$. If $A$ has a unit, then the unit of $\Gamma\mathbb{A}$ is given by the section ${\bf 1}(x):=1_{\mathbb{A}_x}$. In this case $C^{\infty}(M)$ is a unital subalgebra of $\Gamma\mathbb{A}$.
\end{definition}

Our next goal is to endow $\Gamma\mathbb{A}$ with a topology. 

\begin{construction}\label{top on space of sections}
(A topology on $\Gamma\mathbb{A}$). Let $(\mathbb{A},M,A,q)$ be an algebra bundle and $(\varphi_i,U_i)_{i\in I}$ a bundle atlas for $(\mathbb{A},M,A,q)$. We endow $\Gamma\mathbb{A}$ with the initial topology $\mathcal{O}_I$ generated by the maps
\[\Phi_i:\Gamma\mathbb{A}\rightarrow C^{\infty}(U_i,A),\,\,\,s\mapsto s_i:=\pr_A\circ\varphi_i^{-1}\circ s_{\mid U_i}
\]for $i\in I$. Here, the right-hand side carries the smooth compact open topology from Definition \ref{smooth compact open topology}. This topology turns $\Gamma\mathbb{A}$ into a locally convex algebra (cf. Proposition \ref{C(M,A) loc. con. algebra}). We will see soon that the topology on $\Gamma\mathbb{A}$ does not depend on the particular choice of the bundle atlas $(\varphi_i,U_i)_{i\in I}$. 
\end{construction}

\begin{lemma}\label{locally defines sections}
Let $(\mathbb{A},M,A,q)$ be an algebra bundle and $(\varphi_i,U_i)_{i\in I}$ a bundle atlas for $(\mathbb{A},M,A,q)$. Moreover, let $\varphi_{ji}:=\varphi_j^{-1}\circ\varphi_i$ \emph{(}on $(U_i\cap U_j)\times A$\emph{)} and $s_i:=\Phi_i(s)$ for $i,j\in I$ and $s\in\Gamma\mathbb{A}$ \emph{(}cf. Construction \ref{top on space of sections}\emph{)}. Then the following assertions hold:
\begin{itemize}
\item[(a)]
If $s\in\Gamma\mathbb{A}$ and $\widehat{s_i}(x):=(x,s_i(x))$ for $i\in I$ and $x\in U_i$, then $\widehat{s_j}=\varphi_{ji}\circ\widehat{s_i}$ for each $i,j\in I$.
\item[(b)]
Conversely, if $(s_i)_{i\in I}\in\prod_{i\in I}C^{\infty}(U_i,A)$ satisfies $\widehat{s_j}=\varphi_{ji}\circ\widehat{s_i}$ for each $i,j\in I$, then the map 
\[s:M\rightarrow\mathbb{A},\,\,\,s(x):=(\varphi_i\circ \widehat{s_i})(x),\,\,\,x\in U_i,
\]defines a section of the bundle $(\mathbb{A},M,A,q)$ with $\Phi_i(s)=s_i$ for each $i\in I$.
\end{itemize}
\end{lemma}

\begin{proof}
(a) This is just a simple calculation involving the bundle charts $(\varphi_i,U_i)_{i\in I}$.

(b) Let $x\in U_i\cap U_j$. Then
\[(\varphi_j\circ \widehat{s_j})(x)=(\varphi_j\circ \varphi_{ji}\circ\widehat{s_i})(x)=(\varphi_i\circ \widehat{s_i})(x)
\]shows that the map $s$ is well-defined. It obviously defines a section of the bundle $(\mathbb{A},M,A,q)$ with $\Phi_i(s)=s_i$ for each $i\in I$.
\end{proof}

\begin{proposition}\label{imPhi}
Suppose we are in the situation of Lemma \ref{locally defines sections}. Further, consider the algebra homomorphism
\[\Phi_I:(\Gamma\mathbb{A},\mathcal{O}_I)\rightarrow\prod_{i\in I}C^{\infty}(U_i,A),\,\,\,\Phi_I(s):=(\Phi_i(s))_{i\in I}=(s_i)_{i\in I}.
\]Then the following assertions hold:
\begin{itemize}
\item[(a)]
We have
\[\im(\Phi_I)=\{(s_i)_{i\in I}:\,(\forall i,j\in I)\,\widehat{s_j}=\varphi_{ji}\circ\widehat{s_i}\}.
\]
\item[(b)]
The algebra homomorphism $\Phi_I$ is a topological embedding with closed image.
\end{itemize}
\end{proposition}

\begin{proof}
(a) The first assertion is a direct consequence of Lemma \ref{locally defines sections}.

(b) Clearly, the map $\Phi_I$ is injective. Moreover, the topology $O_I$ just defined in Construction \ref{top on space of sections} turns it into a topological embedding. To see that $\im(\Phi_I)$ is closed, let $(s_{\alpha})_{\alpha\in\Lambda}$ be a net in $\Gamma\mathbb{A}$ such that $\Phi_I(s_{\alpha})$ converges to some $(s_i)_{i\in I}\in\prod_{i\in I}C^{\infty}(U_i,A)$. Then $\lim_{\alpha}s_{\alpha,i}(x)=s_i(x)$ for each $i\in I$ and $x\in U_i$. From this we conclude that
\[\widehat{s_j}(x)=\lim_{\alpha}\widehat{s_{\alpha,j}}(x)=\lim_{\alpha}(\varphi_{ji}\circ\widehat{s_{\alpha,i}})(x)=(\varphi_{ji}\circ\widehat{s_i})(x)
\]for all $i,j\in I$ and $x\in U_i\cap U_j$ and thus that $(s_i)_{i\in I}\in\im\Phi_I$ by part (a).
\end{proof}

\begin{remark}\label{sco topology=initial topology}
(Smooth compact open topology$=$initial topology). Given a manifold $M$ and a locally convex algebra $A$, the topology $\mathcal{O}_I$ on $C^{\infty}(M,A)$ induced by an atlas $(\varphi_i,U_i)_{i\in I}$ of $M$ (consider $C^{\infty}(M,A)$ as the sections of the trivial bundle $(M\times A,M,A,\pr_M)$ and use Construction \ref{top on space of sections}) coincides with the smooth compact opnen topology traditionally considered on $C^{\infty}(M,A)$ (cf. Definition \ref{smooth compact open topology}). A very nice proof of this statement (and more background on topologies on function spaces) can be found in \cite[Proposition 4.19 (d)]{Gl04}.
\end{remark}

\begin{theorem}\label{topology unique}
Let $(\mathbb{A},M,A,q)$ be an algebra bundle and $(\varphi_i,U_i)_{i\in I}$ a bundle atlas for $(\mathbb{A},M,A,q)$. Further, let $(\psi_j,V_j)_{j\in J}$ be another bundle atlas for $(\mathbb{A},M,A,q)$. Then the identity map
\[\id:(\Gamma\mathbb{A},\mathcal{O}_I)\rightarrow(\Gamma\mathbb{A},\mathcal{O}_J)
\]is an isomorphism of locally convex algebras. In particular, the topology on $\Gamma\mathbb{A}$ does not depend on the particular choice of the bundle atlas. 
\end{theorem}

\begin{proof}
The universal property of the initial topology $\mathcal{O}_J$ implies that the identity map $\id$ is continuous if and only if the maps
\[\Phi_j:(\Gamma\mathbb{A},\mathcal{O}_I)\rightarrow C^{\infty}(V_j,A),\,\,\,s\mapsto s_j
\]are continuous for each $j\in J$. Therefore, we fix $j\in J$ and note that the continuity of $\Phi_j$ follows from Proposition \ref{imPhi} (b), Remark \ref{sco topology=initial topology} and the continuity of the map
\[\prod_{i\in I}C^{\infty}(U_i,A)\rightarrow\prod_{i\in I} C^{\infty}(V_j\cap U_i,A)\,\,\,(s_i)_{i\in I}\mapsto (g_{ji}\circ(\id_{V_j\cap U_i}\times(s_i)_{\mid V_j}))_{i\in I},
\]where $g_{ji}:(V_j\cap U_i)\times A\rightarrow A$ denotes the smooth map defined by the transition function $\psi_j^{-1}\circ\varphi_i$. A similar argument shows that the ``inverse" map $\id^{-1}$ is continuous. Thus, $\id$ is an isomorphism of locally convex algebras.
\end{proof}

\begin{corollary}\label{A frechet GammaA frechet}
Let $A$ be a Fr\'echet algebra and $(\mathbb{A},M,A,q)$ an algebra bundle. Then $\Gamma\mathbb{A}$ carries a unique structure of a Fr\'echet algebra, when endowed with the topology of Construction \ref{top on space of sections}.
\end{corollary}

\begin{proof}
If $(\varphi_i,U_i)_{i\in I}$ is a countable bundle atlas for $(\mathbb{A},M,A,q)$, then Proposition \ref{imPhi} (b) implies that $\Gamma\mathbb{A}$ carries the structure of a Fr\'echet algebra, since the right-hand side is a Fr\'echet algebra and the image of $\Phi_I$ is closed. That this structure is unique is now a consequence of Theorem \ref{topology unique}.
\end{proof}

The following remark can be used to verify that certain maps to $\Gamma\mathbb{A}$ are smooth. We recall that the space $\prod_{i\in I}C^{\infty}(U_i,A)$ carries the structure of an infinite-dimensional manifold:

\begin{remark}\label{smooth structure on sections}
(Verifying smoothness on $\Gamma\mathbb{A}$). If $(\mathbb{A},M,A,q)$ is an algebra bundle and $(\varphi_i,U_i)_{i\in I}$ a bundle atlas for $(\mathbb{A},M,A,q)$, then the map
\[\Phi_I:\Gamma\mathbb{A}\rightarrow\prod_{i\in I}C^{\infty}(U_i,A),\,\,\,\Phi_I(s):=(\Phi_i(s))_{i\in I}=(s_i)_{i\in I}
\]is a smooth embedding. Indeed, Proposition \ref{imPhi} (b) implies that the map $\Phi_I$ is continuous and linear. Hence, $\Phi_I$ is smooth. Since $\Phi_I$ has closed image, we conclude from \cite[Lemma 2.2.7]{Gl05}, that a map $f:N\rightarrow\Gamma\mathbb{A}$ from a manifold $N$ to $\Gamma\mathbb{A}$ is smooth if and only if the composition $\Phi_I\circ f$ is smooth.
\end{remark}

\section{Smooth Localization of Sections of Algebra Bundles}\label{SLSAB}

In this section we finally want to prove a smooth analogue of Proposition \ref{K(X)_f=O(X_f)} for the algebra $C^{\infty}(M)$ of smooth functions on some manifold $M$. In fact, we show that if $A$ is a unital Fr\'echet algebra, $(\mathbb{A},M,A,q)$ an algebra bundle and $f\in C^{\infty}(M,\mathbb{R})$, then the smooth localization of $\Gamma\mathbb{A}$ with respect to $f$ (cf. Definition \ref{sm. loc.}) is isomorphic (as a unital Fr\'echet algebra) to $\Gamma\mathbb{A}_{M_f}$, where $M_f:=\{m\in M:\,f(m)\neq 0\}$. To be more precise, we show that the map
\[\phi_f:\Gamma\mathbb{A}_{\{f\}}\rightarrow\Gamma\mathbb{A}_{M_f},\,\,\,[F]\mapsto F\circ\left(\frac{1}{f}\times\id_{M_f}\right)
\]is an isomorphism of unital Fr\'echet algebras. We start with a very useful lemma of Hadamard:

\begin{lemma}\label{Hadamards lemma}
\emph{(}Hadamard's lemma\emph{)}. Let $E$ be a complete locally convex space. Further, let $U$ be an open convex subset of $\mathbb{R}^k\times\mathbb{R}^{n-k}$ containing $0$ and $f:U\rightarrow E$ be a smooth function that vanishes on $U\cap(\{0\}\times\mathbb{R}^{n-k})$. If $\pr_j:\mathbb{R}^n\rightarrow\mathbb{R}$ denotes the projection to the j-th factor, then
\[f=\sum_{j=1}^{k}g_j\cdot\pr_j
\]holds for suitable smooth functions $g_j:U\rightarrow E$.
\end{lemma}
\begin{proof}
For $x=(x_1,\ldots,x_n)$ in $U$ we define $y:=(x_1,\ldots,x_k,0)$ and $z:=(0,x_{k+1},\ldots,x_n)$. We further define a smooth $E$-valued curve $h$ on $[0,1]$ by
\[h:[0,1]\rightarrow E,\,\,\, h(t):=f(z+ty).
\]Then $h'(t)=\sum_{j=1}^k\frac{\partial f}{\partial x_j}(z+ty)\cdot x_j$, and for $g_j(x):=\int^1_0\frac{\partial f}{\partial x_j}(z+ty)dt$ we obtain
\[f(x)=f(x)-f(z)=h(1)-h(0)=\int^1_0h'(t)dt=\sum^k_{j=1}g_j(x)\cdot x_j,
\]i.e., $f=\sum_{j=1}^{k}g_j\cdot\pr_j$ as desired.
\end{proof}


\begin{theorem}\label{manifold mod sub=sub}
\emph{(}A factorization theorem for algebra bundles\emph{)}. Let $A$ be a complete locally convex algebra and $(\mathbb{A},M,A,q)$ an algebra bundle. Further, let $h=(h_1,\ldots,h_k):M\rightarrow\mathbb{R}^k$ be a smooth function with $0\in\mathbb{R}^k$ as a regular value. If $H:=h^{-1}(0)$ is the corresponding closed submanifold of $M$, then the restriction map $R_H:\Gamma\mathbb{A}\rightarrow \Gamma\mathbb{A}_H$, $s\mapsto s_{\mid H}$ is a surjective morphism of locally convex algebras. Its kernel $\ker(R_H)$ is equal to
\[\langle h_1,\ldots,h_k\rangle:=\{h_1\cdot s_1+\ldots+h_k\cdot s_k:\,s_1,\ldots s_k\in\Gamma\mathbb{A}\}.
\]In particular, $\langle h_1,\ldots,h_k\rangle$ is a closed two-sided ideal of $\Gamma\mathbb{A}$.
\end{theorem}

\begin{proof}
The proof of this theorem is devided into three parts:
 
(i) To show that the map $R_H$ is a surjective morphism of locally convex algebras, we first note that $R_H$ is obviously linear, multiplicative and continuous as a restriction map. Thus, it remains to prove its surjectivity: For this we have to show that any section $s:H\rightarrow\mathbb{A}_H$ can be extended to a global section on $M$. Locally this can be done, since $H$ looks like $\mathbb{R}^{n-k}$ in $\mathbb{R}^n$ (for $n=\dim M$) and $s$ like a smooth $A$-valued function on $\mathbb{R}^{n-k}$. To get a global extension we have to use a partition of unity. 
 
(ii) Next, we show the equality of the kernel of $R_H$ and the ideal $\langle h_1,\ldots,h_k\rangle$. Since $R_H$ is continuous, this will in particular imply that $\langle h_1,\ldots,h_k\rangle$ is a closed ideal of $\Gamma\mathbb{A}$: We immediately verify that $\langle h_1,\ldots,h_k\rangle\subseteq\ker(R_H)$. For the other inclusion, let $s\in\Gamma\mathbb{A}$ with $R_H(s)=0$, i.e., $s\equiv 0$ on $H$. We claim that each $m\in M$ has an open $m$-neighbourhood $U$ such that
\[s_{\mid U}=\sum^k_{j=1}h_j\cdot s^U_j
\]holds for suitable sections $s^U_j:U\rightarrow\mathbb{A}_U$. If this is the case, we choose an open cover $(U_i)_{i\in I}$ of $M$ such that 
\[s_{\mid U_i}=\sum^k_{j=1}h_j\cdot s^i_j
\]holds for suitable sections $s^i_j:U_i\rightarrow\mathbb{A}_{U_i}$ and a partition of unity $(\psi_i,U_i)_{i\in I}$ subordinated to this cover. Since each section $\psi_i\cdot s^i_j:U_i\rightarrow\mathbb{A}_{U_i}$ can be extended to a section on $M$ (by defining it to be zero outside $\supp(\psi_i)$), we conclude that
\[s=\sum_{i\in I}\psi_i\cdot s=\sum_{i\in I}\left(\sum_{j=1}^k\psi_i\cdot h_j\cdot s^ i_j\right)=\sum^k_{j=1}h_j\left(\sum_{i\in I}\psi_i\cdot s^ i_j\right)\in\langle h_1,\ldots,h_k\rangle.
\]

(iii) We now prove the claim of part (ii). For this, let first be $m\in M\backslash H$. Then there exists $j\in\{1,\ldots,k\}$ and an open $m$-neighbourhood $U$ such that $h_j\neq 0$ on $U$. Therefore, we can write
\[s_{\mid U}=h_j\cdot\frac{s_{\mid U}}{{h_j}_{\mid U}}.
\]Now, let $m\in H$. By the Implicit Function Theorem we may choose a chart $(\varphi,U)$ around $m$ such that $h\circ\varphi^{-1}:\mathbb{R}^n\rightarrow\mathbb{R}^k$ looks like the projection onto the first $k$-coordinates and $s\circ\varphi^{-1}$ like a smooth $A$-valued function on $\mathbb{R}^n$. If we write $\mathbb{R}^n=\mathbb{R}^k\times\mathbb{R}^{n-k}$ and use Lemma \ref{Hadamards lemma}, then we conclude that
\[s\circ\varphi^{-1}=\sum_{j=1}^k(h_j\circ\varphi^{-1})\cdot (s^U_j\circ\varphi^{-1})
\]for suitable sections $s^ U_j:U\rightarrow \mathbb{A}_U$. In particular, from this we immediately obtain the desired representation
\[s_{\mid U}=\sum_{j=1}^k h_j\cdot s^U_j.
\]
\end{proof}

\begin{lemma}\label{algebra section 0}
If $(\mathbb{A},M,A,q)$ is an algebra bundle and $N$ a manifold, then the following assertions hold:
\begin{itemize}
\item[(a)]
If $(\varphi_i,U_i)_{i\in I}$ is a bundle atlas for $(\mathbb{A},M,A,q)$, then the algebra map
\[\widehat{\Phi}_I:C^{\infty}(N,\Gamma\mathbb{A})\rightarrow\prod_{i\in I}C^{\infty}(N\times U_i,A),\,\,\,F\mapsto(\Phi_i\circ F)_{i\in I}
\]is a topological embedding with closed image given by
\[\im(\widehat{\Phi}_I)=\{(F_i)_{i\in I}:\,(\forall i,j\in I, n\in N)\,\widehat{F_j(n)}=\varphi_{ji}\circ\widehat{F_i(n)}\}
\]
\emph{(}cf. Lemma \ref{locally defines sections} and Lemma \ref{smooth exp law} \emph{)}.
\item[(b)]
If $\pr_M:N\times M\rightarrow M$ denotes the canonical projection onto $M$ and $\pr_M^{*}(\mathbb{A})$ the corresponding pull-back bundle, then the map
\[\Psi:C^{\infty}(N,\Gamma\mathbb{A})\rightarrow\Gamma\pr_M^{*}(\mathbb{A}),\,\,\,\Psi(F)(n,m):=(n,m,F(n)(m))
\]is an isomorphism of unital locally convex algebras.
\end{itemize}
\end{lemma}

\begin{proof}
(a) The first assertion immediately follows from [Gl05], Lemma 4.2.6 applied to Proposition \ref{imPhi}.

(b) For the second assertion we first note that both spaces, $C^{\infty}(N,\Gamma\mathbb{A})$ and $\Gamma\pr_M^{*}(\mathbb{A})$, carry the structure of a locally convex algebra and that the map $\Psi$ is an algebras homomorphism. To see that $\Psi$ is bijective it is enough to note that each section $s$ of the pull-back bundle $\pr_M^{*}(\mathbb{A})$ has the form
\[s:N\times M\rightarrow\pr_M^{*}(\mathbb{A}),\,\,\,s(n,m)=(n,m,F_s(n,m))
\]for some smooth function $F_s:N\times M\rightarrow\mathbb{A}$ satisfying $q\circ F_s(n,\cdot)=\id_M$ for each $n\in N$, i.e., for some $F_s\in C^{\infty}(N,\Gamma\mathbb{A})$. It therefore remains to show that the map $\Psi$ is bicontinuous: For this let $(\varphi_i,U_i)_{i\in I}$ be a bundle atlas for $(\mathbb{A},M,A,q)$. Then the bicontinuity of $\Psi$ follows from the definition of the topology on $\Gamma\pr_M^{*}(\mathbb{A})$ (which is induced from the bundle atlas $(\varphi_i,U_i)_{i\in I}$) and part (a) of the lemma.
\end{proof}

\begin{proposition}\label{algebra section 1}
Let $A$ be a unital locally convex algebra and $(\mathbb{A},M,A,q)$ an algebra bundle. If $f_1,\ldots,f_k:M\rightarrow\mathbb{R}$ are smooth functions and $M_{f_1,\ldots, f_k}:=\bigcap_{j=1}^k M_{f_j}$, then the map
\[\Phi_{f_1,\ldots, f_k}:C^{\infty}(\mathbb{R}^k,\Gamma\mathbb{A})\rightarrow\Gamma\mathbb{A}_{M_{f_1,\ldots, f_k}},\,\,\,F\mapsto F\circ\left(\frac{1}{f_1}\times\cdots\times\frac{1}{f_k}\times\id_{M_{f_1,\ldots, f_k}}\right)
\]is a surjective morphism of unital locally convex algebras. Its kernel $\ker(\Phi_{f_1,\ldots, f_k})$ is equal to
\[I_{f_1,\ldots, f_k}:=\{(1-t_1f_1)\cdot g_1+\ldots+(1-t_kf_k)\cdot g_k:\,g_1,\ldots g_k\in C^{\infty}(\mathbb{R}^k,\Gamma\mathbb{A})\}.
\]In particular, $I_{f_1,\ldots, f_k}$ is a closed two-sided ideal of $C^{\infty}(\mathbb{R}^k,\Gamma\mathbb{A})$.
\end{proposition}

\begin{proof}
The proof of this proposition is divided into three parts:

(i) We first note that $0\in\mathbb{R}^k$ is a regular value for the function 
\[h:\mathbb{R}^k\times M\rightarrow\mathbb{R}^k,\,\,\,(t_1,\ldots,t_k,m)\mapsto(1-t_1f_1(m),\ldots,1-t_kf_k(m))
\]In particular, 
\[H:=h^{-1}(0)=\left\{\left(\frac{1}{f_1(x)},\ldots,\frac{1}{f_k(x)},x\right):\,x\in M_{f_1,\ldots, f_k}\right\}
\]is a closed submanifold of $\mathbb{R}^k\times M$. Further, we note that $H$ is diffeomorphic to $M_{f_1,\ldots, f_k}$.

(ii) Next, Lemma \ref{algebra section 0} (b) applied to the algebra bundle $(\mathbb{A},M,A,q)$ and $N=\mathbb{R}^k$ implies that the map
\[\Psi:C^{\infty}(\mathbb{R}^k,\Gamma\mathbb{A})\rightarrow\Gamma\pr_M^{*}(\mathbb{A}),\,\,\,\Psi(F)((r_1,\ldots,r_k),m):=((r_1,\ldots,r_k),m,F(r_1,\ldots,r_k)(m))
\]is an isomorphism of unital locally convex algebras. A similar argument shows that the space $\Gamma\left(\pr_M^{*}(\mathbb{A})\right)_H$ of sections of the pull-back bundle $\pr_M^{*}(\mathbb{A})$ restricted to the submanifold $H$ is isomorphic (as a unital locally convex algebra) to $\Gamma\mathbb{A}_{M_{f_1,\ldots, f_s}}$. 

(iii) The assertions of the proposition now follow from Theorem \ref{manifold mod sub=sub} applied to the restriction map
\[R_H:\Gamma\pr_M^{*}(\mathbb{A})\rightarrow\Gamma\left(\pr_M^{*}(\mathbb{A})\right)_H,\,\,\,s\mapsto s_{\mid H}
\]
\end{proof}

\begin{theorem}\label{algebra section 2}
\emph{(}Smooth localization of sections of algebra bundles\emph{)}. Suppose we are in the situation of Proposition \ref{algebra section 1}. If $A$ is a unital Fr\'{e}chet algebra, then the map $\Phi_{f_1,\ldots, f_k}$ is open. Moreover, the induced map
\[\phi_{f_1,\ldots, f_k}:\Gamma\mathbb{A}_{\{{f_1,\ldots, f_k}\}}\rightarrow\Gamma\mathbb{A}_{M_{{f_1,\ldots, f_k}}},\,\,\,[F]\mapsto F\circ\left(\frac{1}{f_1}\times\cdots\times\frac{1}{f_k}\times\id_{M_{f_1,\ldots, f_k}}\right)
\]is an isomorphism of unital Fr\'echet algebras \emph{(}cf. Definition \ref{sm. loc.} for the definition of $\Gamma\mathbb{A}_{\{f_1,\ldots, f_k\}}$\emph{)}.
\end{theorem}

\begin{proof}
If $A$ is a unital Fr\'{e}chet algebra, then so are $\Gamma\mathbb{A}$, $\Gamma\mathbb{A}_{M_{{f_1,\ldots, f_k}}}$ (cf. Corollary \ref{A frechet GammaA frechet}) and $C^{\infty}(\mathbb{R}^k,\Gamma\mathbb{A})$. In particular, the Open Mapping Theorem (cf. \cite[Chapter III, Section 2.2]{Sch99}) implies that the map $\Phi_{f_1,\ldots, f_k}$ is open. Since the quotient of a Fr\'{e}chet space by a closed subspace is again Fr\'{e}chet (cf. \cite[\S 3.5]{Bou55}), we can use the Open Mapping Theorem again to see that the map $\phi_{f_1,\ldots, f_k}$ is an isomorphism of unital Fr\'echet algebras.
\end{proof}

Although $C^{\infty}(M)_{\{f\}}$ is not the ring of fractions with some power of $f$ as denominator, Theorem \ref{algebra section 2} yields the desired smooth analogue of Proposition \ref{K(X)_f=O(X_f)}:

\begin{corollary}\label{C(M)_f=C(M_f)}
Let $M$ be a manifold and $A$ a unital Fr\'echet algebra. Further, let $f_1,\ldots,f_k:M\rightarrow\mathbb{R}$ be smooth functions and $M_{f_1,\ldots, f_k}:=\bigcap_{j=1}^s M_{f_j}$. Then the map
\[\phi_{f_1,\ldots, f_k}:C^{\infty}(M,A)_{\{f_1,\ldots, f_k\}}\rightarrow C^{\infty}(M_{f_1,\ldots, f_k},A),\,\,\,[F]\mapsto F\circ\left(\frac{1}{f_1}\times\cdots\times\frac{1}{f_k}\times\id_{M_{f_1,\ldots, f_k}}\right)
\]is an isomorphism of unital Fr\'{e}chet algebras.
\end{corollary}

\begin{proof}
The assertion directly follows from Theorem \ref{algebra section 2} applied to the trivial algebra bundle $(M\times A,A,M,q_M)$.
\end{proof}

We now present a very useful and well-known theorem of analysis:

\begin{theorem}\label{whitneys theorem}
\emph{(}Whitneys Theorem\emph{)}. Any open subset $U$ of a manifold $M$ is the complement of the zeroset of a smooth function $f:M\rightarrow\mathbb{R}$, i.e.,
\[U=M_f:=\{m\in M:\,f(m)\neq 0\}\,\,\,\text{for some}\,\,\,f\in C^{\infty}(M,\mathbb{R}).
\]Such a function $f$ is called a $U$-defining function.
\end{theorem}

\begin{proof}
A proof for the case $M=\mathbb{R}^n$ can be found in [GuPo74], Exercise 1.5.11. A combination with a partition-of-unity argument proves the general case.
\end{proof}

\begin{corollary}\label{algebra section 3}
If $A$ is a unital Fr\'echet algebra, $(\mathbb{A},M,A,q)$ an algebra bundle and $f_U$ a $U$-defining function corresponding to an open subset $U$ of $M$, then the map
\[\phi_U:\Gamma\mathbb{A}_{\{f_U\}}\rightarrow\Gamma\mathbb{A}_U,\,\,\,[F]\mapsto F\circ\left(\frac{1}{f_U}\times\id_U\right)
\]is an isomorphism of unital Fr\'echet algebras.
\end{corollary}

\begin{proof}
This assertion is a direct consequence of Theorem \ref{algebra section 2}.
\end{proof}

\section{The Spectrum of the Space of Sections of an Algebra Bundle\\with Finite-Dimensional Commutative Fibre}\label{TSSSABFDCF}

In the last two sections we were concerned with sections of algebra bundles. The goal of this section is to describe the spectrum of the algebra of sections of an algebra bundle with finite-dimensional commutative fibre. We show that if $(\mathbb{A},M,A,q)$ is such an algebra bundle, then the spectrum of the corresponding algebra $\Gamma\mathbb{A}$ of sections is an $n$-fold covering of $M$ (here $n=\dim A_{\text{ss}}$).

\begin{theorem}\label{Wedderburn-Artin}
\emph{(}Wedderburn-Artin\emph{)}. Each finite-dimensional semisimple unital algebra is isomorphic to a direct product of matrix algebras, i.e., to
\[\M_{k_1}(\mathbb{C})\times\cdots\M_{k_n}(\mathbb{C})\,\,\,\text{for some}\,\,\,k_1,\ldots k_n\in\mathbb{N}.
\]Conversely, each finite direct product of matrix algebras is a semisimple unital algebra.
\end{theorem}

\begin{proof}
For the proof of this statement we refer, for example, to \cite[Theorem 2.4.3]{DrKi94}.
\end{proof}

\begin{corollary}\label{TSSSABFDCF1}
Each finite-dimensional semisimple unital algebra which is commutative is isomorphic to a direct product of copies of $\mathbb{C}$.
\end{corollary}

\begin{proof}
This assertion is a direct consequence of Theorem \ref{Wedderburn-Artin}.
\end{proof}

If a unital algebra $A$ is not semisimple, it is quite natural to measure the defect for $A$ not being semisimple. This can be done with the so-called Jacobson radical:

\begin{definition}\label{The Jacobson radical}
(The Jacobson radical). The \emph{Jacobson radical} $A_{\text{n}}$ of a finite-dimensional unital algebra $A$ can be defined as
\begin{itemize}
\item[(a)]
the largest nilpotent ideal of $A$.
\item[(b)]
the intersection of all maximal ideals of $A$.
\item[(c)]
the smallest ideal of $A$ such that the quotient algebra $A/A_{\text{n}}$ is semisimple.
\end{itemize}
In fact, for the equivalence of these definition we refer to \cite[Section 3.1]{DrKi94}.
\end{definition}

\begin{lemma}\label{TSSSABFDCF2}
Let $A$ be a finite-dimensional unital algebra. Further, let $A_{\emph{\text{n}}}$ be the Jacobson radical of $A$ and $A_{\emph{\text{ss}}}:=A/A_{\emph{\text{n}}}$ the corresponding semisimple unital quotient algebra. Then the following assertions hold:
\begin{itemize}
\item[\emph{(a)}]
Each algebra automorphism $\phi\in\Aut(A)$ leaves $A_{\emph{\text{n}}}$ invariant, i.e., restricts to an algebra automorphism $\phi_{\emph{\text{n}}}$ of $A_{\emph{\text{n}}}$. We write 
\[\pi_{\emph{\text{n}}}:\Aut(A)\rightarrow\Aut(A_{\emph{\text{n}}}),\,\,\,\phi\mapsto\phi_{\emph{\text{n}}}
\]for the corresponding group homomorphism.
\item[\emph{(b)}]
In particular, each algebra automorphism $\phi\in\Aut(A)$ descends to an algebra automorphism $\phi_{\emph{\text{ss}}}$ of $A_{\emph{\text{ss}}}$. We write 
\[\pi_{\emph{\text{ss}}}:\Aut(A)\rightarrow\Aut(A_{\emph{\text{ss}}}),\,\,\,\phi\mapsto\phi_{\emph{\text{ss}}}
\]for the corresponding group homomorphism.
\end{itemize}
\end{lemma}


\begin{proof}
(a) The first assertion follows from the Definition of $A_{\text{n}}$ (cf. Definition \ref{The Jacobson radical} (a) or (b)).

(b) The second assertion is a consequence of part (a).
\end{proof}

\begin{proposition}\label{TSSSABFDCF3}
Suppose we are in the situation of Lemma \ref{TSSSABFDCF2}. Further, let $(\mathbb{A},M,A,q)$ be an algebra bundle and $(\varphi_i,U_i)_{i \in I}$ be a bundle atlas for $(\mathbb{A},M,A,q)$. If $U_{ij}:=U_i\cap U_j$ for $i,j\in I$ and $g_{ij}\in C^{\infty}(U_{ij},\Aut(A))$ are the corresponding transition functions, then the following assertions hold:
\begin{itemize}
\item[(a)]
There exists an algebra \emph{(}sub-\emph{)} bundle $(\mathbb{A}_{\emph{\text{n}}},M,A_{\emph{\text{n}}},q_{\emph{\text{n}}})$ of $(\mathbb{A},M,A,q)$ with bundle charts $\varphi^ {\emph{\text{n}}}_i:U_i\times A_{\emph{\text{n}}}\rightarrow(\mathbb{A}_{\emph{\text{n}}})_{U_i}$ such that
\[((\varphi^{\emph{\text{n}}}_i)^{-1}\circ\varphi^{\emph{\text{n}}}_j)(x,a_{\emph{\text{n}}})=(x,g^{\emph{\text{n}}}_{ij}(x).a_{\emph{\text{n}}}),\,\,\,\text{where}\,\,\,g^ {\emph{\text{n}}}_{ij}:=\pi_{\emph{\text{n}}}\circ g_{ij}\in C^{\infty}(U_{ij},\Aut(A_{\emph{\text{n}}})).
\]
\item[(b)]
There exists an algebra bundle $(\mathbb{A}_{\emph{\text{ss}}},M,A_{\emph{\text{ss}}},q_{\emph{\text{ss}}})$ with bundle charts \mbox{$\varphi^{\emph{\text{ss}}}_i:U_i\times A_{\emph{\text{ss}}}\rightarrow(\mathbb{A}_{\emph{\text{ss}}})_{U_i}$} such that
\[((\varphi^{\emph{\text{ss}}}_i)^{-1}\circ\varphi^{\emph{\text{ss}}}_j)(x,a_{\emph{\text{ss}}})=(x,g^{\emph{\text{ss}}}_{ij}(x).a_{\emph{\text{ss}}}),\,\,\,\text{where}\,\,\,g^ {\emph{\text{ss}}}_{ij}:=\pi_{\emph{\text{ss}}}\circ g_{ij}\in C^{\infty}(U_{ij},\Aut(A_{\emph{\text{ss}}})).
\]In particular, the algebra bundle $(\mathbb{A}_{\emph{\text{ss}}},M,A_{\emph{\text{ss}}},q_{\emph{\text{ss}}})$ can be realized as the quotient algebra bundle of $(\mathbb{A},M,A,q)$ with respect to $(\mathbb{A}_{\emph{\text{n}}},M,A_{\emph{\text{n}}},q_{\emph{\text{n}}})$. The corresponding quotient map is given by
\[p:\mathbb{A}\rightarrow\mathbb{A}_{\emph{\text{ss}}},\,\,\,p(a):=a+(\mathbb{A}_{\emph{\text{n}}})_m\,\,\,\text{for}\,\,\,a\in\mathbb{A}_m.
\]
\item[(c)]
We obtain a short exact sequence
\begin{align}
0\longrightarrow\Gamma\mathbb{A}_{\emph{\text{n}}}\longrightarrow\Gamma\mathbb{A}\stackrel{\Gamma p}{\longrightarrow} \Gamma\mathbb{A}_{\emph{\text{ss}}}\longrightarrow 0\notag
\end{align}
of \emph{(}Fr\'echet-\emph{)} algebras. Here, the second arrow denotes the natural inclusion and 
\[(\Gamma p(s))(m):=p(s(m))\,\,\,\text{for}\,\,\,s\in\Gamma\mathbb{A}\,\,\,\text{and}\,\,\,m\in M.
\]
\end{itemize}
\end{proposition}

\begin{proof}
(a),(b) The first two assertions are a consequence of Lemma \ref{TSSSABFDCF2} and \cite[Part I, Section 3, Theorem 3.2]{St51}.

(c) The third assertion follows from an easy calculation and part (a) and (b).
\end{proof}

\begin{remark}\label{finite open cover}
(Finite open cover property). If $(\mathbb{A},M,A,q)$ is an algebra bundle over a compact manifold $M$, then it is obviously true that there exists a \emph{finite} open cover $\{U_i\}_{1\leq i\leq m}$ of $M$ such that each $\mathbb{A}_{U_i}:=q^{-1}(U_i)$ is trivial. Note that this property is also true for \emph{arbitrary} manifolds: 

Indeed \cite[Theorem 7.5.16]{Con93} states that if $(\mathbb{A},M,A,q)$ is an algebra bundle over a manifold $M$ of dimension $n$, then there is a cover $\{U_i\}_{1\leq i\leq n+1}$ of $M$, such that $\mathbb{A}_{U_i}:=q^{-1}(U_i)$ is trivial. 
\end{remark}

\begin{proposition}\label{TSSSABFDCF4}
Let $(\mathbb{A}_{\emph{\text{n}}},M,A_{\emph{\text{n}}},q_{\emph{\text{n}}})$ be the algebra \emph{(}sub-\emph{)} bundle of Proposition \ref{TSSSABFDCF3} \emph{(}a\emph{)}. Then $\Gamma\mathbb{A}_{\emph{\text{n}}}$ consists of nilpotent elements, i.e., for each $s\in \Gamma\mathbb{A}_{\emph{\text{n}}}$ there exists a natural number $n\in\mathbb{N}$ such that $s^n=0$.
\end{proposition}

\begin{proof}
In view of Remark \ref{finite open cover} and Proposition \ref{imPhi}, $\Gamma\mathbb{A}_{\text{n}}$ can be embedded into a finite product of the form $\prod_{i\in I}C^{\infty}(U_i,A)$. Since each function $s_i\in C^{\infty}(U_i,A)$ appearing in the image $\Phi_I(s)=(s_i)_{i\in I}$ of an element $s\in\Gamma\mathbb{A}_{\text{n}}$ takes values in $A_{\text{n}}$, there exists a natural number $n\in\mathbb{N}$ such that $\Phi_I(s^n)=(s^n_i)_{i\in I}=(0)_{i\in I}$. Thus, we obtain $s^n=0$ as desired.
\end{proof}

\begin{theorem}\label{TSSSABFDCF5}
Let $A$ be a finite-dimensional unital algebra and $(\mathbb{A},M,A,q)$ be an algebra bundle. Further, let $(\mathbb{A}_{\emph{\text{ss}}},M,A_{\emph{\text{ss}}},q_{\emph{\text{ss}}})$ be the algebra bundle of Proposition \ref{TSSSABFDCF3} \emph{(}b\emph{)}. Then the map
\[\Phi:\Gamma_{\Gamma\mathbb{A}}\rightarrow\Gamma_{\Gamma\mathbb{A}_{\emph{\text{ss}}}},\,\,\,\Phi(\chi)(\overline{s}):=\chi(s),
\]where $s\in\Gamma\mathbb{A}$ with $\Gamma p(s)=\overline{s}$, is a well-defined homeomorphism with inverse given by
\[\Psi:\Gamma_{\Gamma\mathbb{A}_{\emph{\text{ss}}}}\rightarrow\Gamma_{\Gamma\mathbb{A}},\,\,\,\Psi(\chi)(s):=\chi(\overline{s})
\]for $\overline{s}=\Gamma p(s)$.
\end{theorem}

\begin{proof}
(i) We first show that $\Phi$ is well-defined. For this we have to verify that each character $\chi\in\Gamma_{\Gamma\mathbb{A}}$ vanishes on $\Gamma\mathbb{A}_{\text{n}}$ (cf. Proposition \ref{TSSSABFDCF3} (c)): Indeed, if $\chi\in\Gamma_{\Gamma\mathbb{A}}$ and $s\in\Gamma\mathbb{A}_{\text{n}}$, then Proposition \ref{TSSSABFDCF4} implies that there exists a natural number $n\in\mathbb{N}$ such that $s^n=0$. Hence, $\chi(s)=0$ follows from $0=\chi(s^n)=\chi(s)^n$.

(ii) Now, a short calculation shows that both maps $\Phi$ and $\Psi$ are continuous and inverse to each other. 
\end{proof}

In the following we describe the spectrum of the algebra of sections of an algebra bundle with finite-dimensional semisimple commutative fibre:

\begin{lemma}\label{TSSSABFDCF6}
Let $M$ be a manifold and $A=\mathbb{C}^n$ for some $n\in\mathbb{N}$. Then the map
\[\Phi:M\times\{1,\ldots,n\}\rightarrow\Gamma_{C^{\infty}(M,A)},\,\,\,\Phi(m,k)(F)=F_k(m),
\]where $F=(F_1,\ldots,F_n)\in C^{\infty}(M,A)$, is a homeomorphism.
\end{lemma}

\begin{proof}
The assertion immediately follows from Theorem \ref{spec of C(M,A) top} and the observation that \mbox{$\Gamma_{\mathbb{C}}=\{\id_{\mathbb{C}}\}$}.
\end{proof}

\begin{lemma}\label{TSSSABFDCF7}
For $n\in\mathbb{N}$ let $I_n:=\{1,\ldots,n\}$ and $A=\mathbb{C}^n$. Then the following assertions hold:
\begin{itemize}
\item[(a)]
We have $\Diff(I_n)=S_n$, where $S_n$ denotes the symmetric group of order $n$. 
\item[(b)]
The map
\[\iota:S_n\rightarrow\Aut(A),\,\,\,\iota(\sigma)(z_1,\ldots,z_n):=(z_{\sigma^{-1}(1)},\ldots,z_{\sigma^{-1}(n)})
\]is an isomorphism of groups.
\item[(c)]
Let $M$ be a manifold. If we identify $A$ with $C^{\infty}(I_n)$, then the map
\[\kappa_M:C^{\infty}(M,A)\rightarrow C^{\infty}(M\times I_n,),\,\,\,f\mapsto f^{\wedge},
\]where
\[f^{\wedge}:M\times I_n\rightarrow \mathbb{C},\,\,\,f^{\wedge}(m,k):=f(m)(k)
\]is an isomorphism of unital Fr\'echet algebras.
\end{itemize}
\end{lemma}

\begin{proof}
(a),(b) The proof of the first two statements is left as an easy exercise to the reader.

(c) The last assertion follows from Lemma \ref{smooth exp law} applied to $N=I_n$.
\end{proof}

\begin{proposition}\label{TSSSABFDCF8}
Let $A$ be a finite-dimensional commutative unital algebra and \mbox{$n:=\dim A_{\emph{\text{ss}}}$}. Further, let $(\mathbb{A},M,A,q)$ be an algebra bundle and $(\varphi_i,U_i)_{i \in I}$ be a bundle atlas for $(\mathbb{A},M,A,q)$. If $U_{ij}:=U_i\cap U_j$ for $i,j\in I$ and $g^{\emph{\text{ss}}}_{ij}\in C^{\infty}(U_{ij},\Aut(A_{\emph{\text{ss}}}))$ are the corresponding transition functions of $(\mathbb{A}_{\emph{\text{ss}}},M,A_{\emph{\text{ss}}},q_{\emph{\text{ss}}})$ \emph{(}cf. Proposition \ref{TSSSABFDCF3}\emph{)}, then there exists a fibre bundle $(\widehat{M},M,I_n,p)$, i.e., an $n$-fold covering $p:\widehat{M}\rightarrow M$ with bundle charts $\psi_i:U_i\times I_n\rightarrow\widehat{M}_{U_i}$ such that
\[(\psi_i^{-1}\circ\psi_j)(x,k)=(x,h_{ij}(x).a),\,\,\,\text{where}\,\,\,h_{ij}:=\iota^{-1}\circ g^{\emph{\text{ss}}}_{ij}\in C^{\infty}(U_{ij},\Diff(I_n))
\]\emph{(}cf. Lemma \ref{TSSSABFDCF7}\emph{)}.
\end{proposition}

\begin{proof}
The assertion of this proposition follows from Lemma \ref{TSSSABFDCF7} (b) and \cite[Part I, Section 3, Theorem 3.2]{St51}.
\end{proof}

\begin{proposition}\label{TSSSABFDCF9}
Suppose we are in the situation of Proposition \ref{TSSSABFDCF8}. Further, consider the algebra homomorphism
\[\Psi_I:C^{\infty}(\widehat{M})\rightarrow\prod_{i\in I}C^{\infty}(U_i\times I_n),\,\,\,\Psi_I(f):=(f_i)_{i\in I}
\]with $f_i:=f_{\mid \widehat{M}_{U_i}}\circ\psi_i$. Then the following assertions hold:
\begin{itemize}
\item[(a)]
If $\psi_{ij}:=\psi_i^{-1}\circ\psi_j$ \emph{(}on $U_{ij}\times I_n$\emph{)}, then
\[\im(\Psi_I)=\{(f_i)_{i\in I}:\,(\forall i,j\in I)\,f_j=\psi_{ij}\circ f_i\}.
\]
\item[(b)]
The algebra homomorphism $\Psi_I$ is a topological embedding with closed image.
\end{itemize}
\end{proposition}

\begin{proof}
The proof of this proposition is similar to the proof of Proposition \ref{imPhi}.
\end{proof}

\begin{proposition}\label{TSSSABFDCF10}
Suppose we are in the situation of Proposition \ref{TSSSABFDCF8}. Further, let 
\[\Phi_I:\Gamma\mathbb{A}_{\emph{\text{ss}}}\rightarrow\prod_{i\in I}C^{\infty}(U_i,A_{\emph{\text{ss}}}),\,\,\,\Phi_I(s):=(\Phi_i(s))_{i\in I}=(s_i)_{i\in I}
\]be the embedding of Proposition \ref{imPhi} applied to the algebra bundle $(\mathbb{A}_{\emph{\text{ss}}},M,A_{\emph{\text{ss}}},q_{\emph{\text{ss}}})$ and $\Psi_I$ the embedding of Proposition \ref{TSSSABFDCF9}. Then the map
\[\Omega_I:\Gamma\mathbb{A}_{\emph{\text{ss}}}\rightarrow C^{\infty}(\widehat{M}),\,\,\,\Omega_I:=\Psi_I^{-1}\circ\left(\prod_{i\in I}\kappa_{U_i}\right)\circ\Phi_I
\]is an isomorphism of unital Fr\'echet algebras. Here, $\kappa_{U_i}$ denotes the map of Lemma \ref{TSSSABFDCF7} \emph{(}c\emph{)} applied to $M=U_i$ and $A_{\emph{\text{ss}}}\cong\mathbb{C}^n$.
\end{proposition}

\begin{proof}
In order to prove the assertion, it is enough to verify that the image of the map $\left(\prod_{i\in I}\kappa_{U_i}\right)\circ\Phi_I$ is equal to $\im(\Psi_I)$. In fact, this immediately follows from Lemma \ref{TSSSABFDCF7} (a) and (b).
\end{proof}

\begin{theorem}\label{TSSSABFDCF11}
\emph{(}The spectrum as an $n$-fold covering\emph{)}. Let $A$ be a finite-dimensional commutative unital algebra. Further, let $(\mathbb{A},M,A,q)$ be an algebra bundle and $(\varphi_i,U_i)_{i \in I}$ be a bundle atlas for $(\mathbb{A},M,A,q)$. If $\widehat{M}$ is as in Proposition \ref{TSSSABFDCF8} and $\delta_{\widehat{m}}$ denotes the usual evaluation map in some $\widehat{m}\in\widehat{M}$, then the map
\[\Xi_I:\widehat{M}\rightarrow\Gamma_{\Gamma\mathbb{A}},\,\,\,\Xi_I(\widehat{m})(s):=\delta_{\widehat{m}}(\Omega_I(\overline{s}))
\]is a homeomorphism. Here, $\Omega_I$ is the map of Proposition \ref{TSSSABFDCF10} and $\overline{s}=\Gamma p(s)$ \emph{(}cf. Proposition \ref{TSSSABFDCF3}\emph{)}. In particular, if $n:=\dim A_{\emph{\text{ss}}}$, then the spectrum $\Gamma_{\Gamma\mathbb{A}}$ is an $n$-fold covering of $M$. 
\end{theorem}

\begin{proof}
We first note that the map $\Xi_I$ is well-defined. Next, let 
\[\Phi_{\widehat{M}}:\widehat{M}\rightarrow \Gamma_{C^{\infty}(\widehat{M})},\,\,\,\widehat{m}\mapsto \delta_{\widehat{m}}\notag
\]be the homeomorphism of Proposition \ref{spec of C(M) top} and
\[\Gamma_{\Omega_I}:\Gamma_{C^{\infty}(\widehat{M})}\rightarrow\Gamma_{\Gamma\mathbb{A}_{\emph{\text{ss}}}},\,\,\,\Gamma_{\Omega_I}(\delta_{\widehat{m}}):=\delta_{\widehat{m}}\circ\Omega_I
\]the homeomorphism induced by the map $\Omega_I$ of Proposition \ref{TSSSABFDCF10}. If $\Psi$ is the homeomorphsim of Theorem \ref{TSSSABFDCF5}, then a short observation shows that $\Xi_I=\Psi\circ\Gamma_{\Omega_I}\circ\Phi_{\widehat{M}}$. From this we conclude that $\Xi$ is a homeomorphism as it is a composition of homeomorphisms.
\end{proof}

\begin{corollary}\label{TSSSABFDCF12}
Suppose we are in the situation of Theorem \ref{TSSSABFDCF11} with $A=\mathbb{C}^n$. Then the map
\[\Xi_I:\widehat{M}\rightarrow\Gamma_{\Gamma\mathbb{A}},\,\,\,\Phi(\widehat{m})(s):=(s_i)_k(m),
\]for $i\in I$ with $m=p(\widehat{m})\in U_i$, $k\in I_n$ with $\psi_i(m,k)=\widehat{m}$ and $s_i=((s_i)_1,\ldots,(s_i)_n)$, is a homeomorphism.

\end{corollary}

\begin{proof}
This assertion is a direct consequence of Theorem \ref{TSSSABFDCF11}.
\end{proof}

\section{Smooth Localization of Dynamical Systems}\label{ACNCPT^nB}

In this section we present a method of localizing dynamical systems $(A,G,\alpha)$ with respect to elements of the commutative fixed point algebra of the induced action of $G$ on the center $C_A$ of $A$. In fact, this construction turns out to be the starting point for our approach to noncommutative principal torus bundles.

\begin{notation} 
Let $(A,G,\alpha)$ be a dynamical system and $C_A$ be the center of $A$. If $c\in C_A, g\in G$ and $a\in A$, then the equation $\alpha(g,c)a=a\alpha(g,c)$ shows that the map $\alpha$ restricts to a continuous action of $G$ on $C_A$ by algebra automorphisms. In the following we write $Z$ for the fixed point algebra of the induced action of $G$ on $C_A$, i.e., $Z=C_A^G$. We further write $A_{\{z\}}$ for the smooth localization of $A$ with respect to $z\in Z$ and 
\[\pi_{\{z\}}:C^{\infty}(\mathbb{R},A)\rightarrow A_{\{z\}}
\]for the corresponding continuous quotient homomorphism (cf. Definition \ref{sm. loc.}). 
\end{notation}

\begin{lemma}\label{A_z is locally convex space again}
With the notations from above, the following assertion holds:
The map 
\[\rho_{\{z\}}:A_{\{z\}}\times Z\rightarrow A_{\{z\}},\,\,\,([f],z)\mapsto [z\cdot f]
\]defines on $A_{\{z\}}$ the structure of a right locally convex $Z$-module.
\end{lemma}

\begin{proof}
Since $Z$ is commutative, the assertion directly follows from Lemma \ref{A_{a...} is a (left) A-modul}.
\end{proof}

\begin{proposition}\label{A_z is algebra}
The locally convex space $A_{\{z\}}$ becomes a locally convex algebra, when equipped with the multiplication
\[m_{\{z\}}:A_{\{z\}}\times A_{\{z\}}\rightarrow A_{\{z\}},\,\,\,m([f],[f']):=[f\cdot f'].
\]
\end{proposition}

\begin{proof}
This assertion is an easy consequence of the construction of $A_{\{z\}}$. Indeed, if $m$ denotes the (continuous) multiplication map of the algebra $C^{\infty}(\mathbb{R},A)$, then the continuity of $m_{\{z\}}$ follows from $m_{\{z\}}\circ(\pi_{\{z\}}\times\pi_{\{z\}})=\pi_{\{z\}}\circ m$ and the fact that the map $\pi_{\{z\}}\times\pi_{\{z\}}$ is continuous, surjective and open.
\end{proof}



\begin{remark}\label{remark on smooth dynamical systems}
(Smooth dynamical systems). We call a dynamical system $(A,G,\alpha)$ \emph{smooth} if $G$ is a Lie group and the group homomorphism $\alpha:G\rightarrow\Aut(A)$ induces a smooth action of $G$ on $A$.

A classical example of such a smooth dynamical system is induced by a smooth action $\sigma:M\times G\rightarrow M$ of a Lie group $G$ on a manifold $M$. In particular, each principal bundle $(P,M,G,q,\sigma)$ induces
a smooth dynamical system $(C^{\infty}(P),G,\alpha)$, consisting of the Fr\'echet algebra of smooth functions on the total space $P$, the structure group $G$ and a group homomorphism $\alpha:G\rightarrow\Aut(C^{\infty}(P))$, induced by the smooth action $\sigma:P\times G\rightarrow P$ of $G$ on $P$. For this and further (non-classical) examples of smooth dynamical systems we refer the interested reader to the examples of Section \ref{a geometric approach to NCPB} or \cite{Wa11a}.
\end{remark}

\begin{lemma}\label{dynamical system and trivial algebra bundles}
Let $(A,G,\alpha)$ be a \emph{(}smooth\emph{)} dynamical system and $M$ a manifold. Then the map
\[\beta:G\times C^{\infty}(M,A)\rightarrow C^{\infty}(M,A),\,\,\,(g,f)\mapsto\alpha(g)\circ f
\]defines a \emph{(}smooth\emph{)} continuous action of $G$ on $C^{\infty}(M,A)$ by algebra automorphisms. In particular, the triple $(C^{\infty}(M,A),G,\beta)$ is a \emph{(}smooth\emph{)} dynamical system.
\end{lemma}

\begin{proof}
We first note that the map
\[\ev_M:C^{\infty}(M,A)\times M\rightarrow A,\,\,\,(f,m)\mapsto f(m)
\]is smooth (cf. \cite[Proposition I.2]{NeWa07}). Then Lemma \ref{smooth exp law} implies that the map $\beta$ is smooth if and only if the map
\[\beta^{\wedge}:G\times C^{\infty}(M,A)\times M\rightarrow A,\,\,\,(g,m,f)\mapsto\alpha(g,f(m))
\]is smooth. Since $\beta^{\wedge}=\alpha\circ(\id_G\times\ev_M)$, we conclude that $\beta^{\wedge}$ is smooth as a composition of smooth maps. Clearly, $\beta$ defines an action of $G$ on $C^{\infty}(M,A)$ by algebra automorphisms.
\end{proof}

\begin{proposition}\label{T^n action on spec again}
If $(A,G,\alpha)$ is a dynamical system and $z\in Z$, then the map
\[\alpha_{\{z\}}:G\times A_{\{z\}}\rightarrow A_{\{z\}},\,\,\,(g,[f]\underline{})\mapsto[\alpha(g)\circ f]
\]defines a continuous action of $G$ on $A_{\{z\}}$ by algebra automorphisms. In particular, the triple $(A_{\{z\}},G,\alpha_{\{z\}})$ is a dynamical system.
\end{proposition}

\begin{proof}
If $g\in G$, then $\alpha(g)(1_A-tz)=1_A-tz$ implies that the map $\alpha_{\{z\}}$ is well-defined. Further, a simple calculation shows that the map $\alpha_{\{z\}}$ defines an action of $G$ on $A_{\{z\}}$ by algebra automorphisms. That this action is continuous is a consequence of the continuity of the map 
\[\beta:G\times C^{\infty}(\mathbb{R},A)\rightarrow C^{\infty}(\mathbb{R},A),\,\,\,(g,f)\mapsto\alpha(g)\circ f
\](cf. Lemma \ref{dynamical system and trivial algebra bundles} applied to $M=\mathbb{R}$), $\alpha_{\{z\}}\circ(\id_G\times\pi_{\{z\}})=\pi_{\{z\}}\circ\beta$ and the fact that $\id_G\times\pi_{\{z\}}$ is continuous, surjective and open.
\end{proof}

\begin{definition}\label{smooth localized dynamical system}
(Smooth localization of dynamical systems). Let $(A,G,\alpha)$ be a dynamical system and $z\in Z$. We call the dynamical system $(A_{\{z\}},G,\alpha_{\{z\}})$ of Proposition \ref{T^n action on spec again} the \emph{smooth localization of $(A,G,\alpha)$} associated to the element $z\in Z$.
\end{definition}

\begin{example}\label{example on smooth localization}
Let $(P,M,G,q,\sigma)$ be a principal bundle and $(C^{\infty}(P),G,\alpha)$ be the induced smooth dynamical system (cf. Remark \ref{remark on smooth dynamical systems}). Further, let $U$ be an open subset of $M$ such that $P_U$ is a trivial principal $G$-bundle over $U$, i.e., such that $P_U\cong U\times G$ holds as $G$-manifolds. If $f$ is a $U$-defining function (cf. Theorem \ref{whitneys theorem}), then $h:=f\circ q$ is a $P_U$-defining function lying in the subalgebra $C^{\infty}(P)^G$. Thus, we conclude from Corollary \ref{C(M)_f=C(M_f)} that the map
\[\phi_U:C^{\infty}(P)_{\{h\}}\rightarrow C^{\infty}(P_U),\,\,\,[F]\mapsto\left(p\mapsto F\left(\frac{1}{h(p)},p\right)\right)
\]is an isomorphism of unital Fr\'{e}chet algebras. Moreover, the map $\phi_U$ is $G$-equivariant. In fact, we have
\[(\phi_U\circ\alpha_{\{h\}})(g,[F])=\alpha(g,\phi_U([F]))
\]for all $g\in G$ and $[F]\in C^{\infty}(P)_{\{h\}}$. In particular, the corresponding localized dynamical system $(C^{\infty}(P)_{\{h\}},G,\alpha_{\{h\}})$ of Proposition \ref{T^n action on spec again} carries the structure of a smooth dynamical system, since it is isomorphic to the smooth dynamical system $(C^{\infty}(P_U),G,\alpha_U)$ corresponding to the trivial principal $G$-bundle $P_U$ over $U$.
\end{example}

In the remaining part of this section we will be concerned with the spectrum of $A_{\{z\}}$:

\begin{lemma}\label{T^n action on spec II again}
If $(A,G,\alpha)$ is a dynamical system and $z\in Z$, then the map 
\[\sigma_{\{z\}}:\Gamma^{\emph{cont}}_{A_{\{z\}}}\times G\rightarrow\Gamma^{\emph{cont}}_{A_{\{z\}}},\,\,\,(\chi,g)\mapsto \chi\circ\alpha_{\{z\}}(g)
\]defines an action of $G$ on $\Gamma^{\emph{cont}}_{A_{\{z\}}}$.
\end{lemma}

\begin{proof}
This claim follows from a simple calculation and is therefore left to the reader.
\end{proof}

\begin{remark}\label{spec A_z}
($\Gamma^{\text{cont}}_{A_{\{z\}}}$). Let $A$ be a complete unital locally convex algebra such that $\Gamma^{\text{cont}}_A$ is locally equicontinuous. Then Theorem \ref{spec A_a} implies that the map
\[\Phi_{\{z\}}:D_A(z)\rightarrow\Gamma^{\text{cont}}_{A_{\{z\}}},\,\,\,\Phi(\chi)([f]):=\chi\left(f\left(\frac{1}{\chi(z)}\right)\right)
\]is a homeomorphism.
\end{remark}

\section{Noncommutative Principal Torus Bundles}\label{a geometric approach to NCPB}

The main goal of this section is to present a geometrically oriented approach to the noncommutative geometry of principal torus bundles. As already mentioned in the Introduction, a trivial noncommutative principal torus bundles is a dynamical system $(A,\mathbb{T}^n,\alpha)$ with the additional property that each isotypic component contains an invertible element. Since the freeness property of a group action is a local condition (cf. Remark \ref{free=locally free}), our main idea is inspired by the classical setting: Loosely speaking, a dynamical system $(A,\mathbb{T}^n,\alpha)$ is called a noncommutative principal $\mathbb{T}^n$-bundle, if it is ``locally" a trivial noncommutative principal $\mathbb{T}^n$-bundle, i.e., Section \ref{ACNCPT^nB} enters the picture. We prove that this approach extends the classical theory of principal torus bundles and present some noncommutative examples. Indeed, we first show that each trivial noncommutative principal torus bundle carries the structure of a noncommutative principal torus bundle in its own right. We further show that examples are provided by sections of algebra bundles with trivial noncommutative principal torus bundle as fibre, sections of algebra bundles which are pull-backs of principal torus bundles and sections of trivial equivariant algebra bundles. At the end of this section we present a very concrete example.  


\begin{notation}
Given a dynamical system $(A,G,\alpha)$, we (again) write $Z$ for the fixed point algebra of the induced action of $G$ on the center $C_A$ of $A$, i.e., $Z:=C_A^G$. In particular, if $P$ is a manifold, $G$ a Lie group group and $(C^{\infty}(P),G,\alpha)$ a (smooth) dynamical system, then $Z=C^{\infty}(P)^G$.
\end{notation}

\begin{definition}\label{NCPT^nB again}
(Noncommutative principal torus bundles). We call a (smooth) dynamical system $(A,\mathbb{T}^n,\alpha)$ a \emph{\emph{(}smooth\emph{)} noncommutative principal $\mathbb{T}^n$-bundle} if for each $\chi\in\Gamma_Z$ there exists an element $z\in Z$ with $\chi(z)\neq 0$ such that the corresponding localized dynamical system $(A_{\{z\}},\mathbb{T}^n,\alpha_{\{z\}})$ of Proposition \ref{T^n action on spec again} is a (smooth) trivial noncommutative principal $\mathbb{T}^n$-bundle (cf. Definition \ref{Trivial NCP T^n-Bundles}). 
\end{definition}

\begin{remark}
The previous definition of noncommutative principal torus bundles is inspired by the classical setting. In fact, we will see in Theorem \ref{NCT^nB for manifold again} that it actually agrees with the classical definition of principal torus bundles if we identify open subsets $U$ of the base manifold with the corresponding $U$-defining functions and use Corollary \ref{C(M)_f=C(M_f)}.
\end{remark}

\begin{remark}\label{remark on NCPT^nB again}
(An equivalent point of view). Given $z\in Z$, we recall that 
\[D(z):=\{\chi\in\Gamma_Z:\,\chi(z)\neq 0\}.
\]Then a short observation shows that a (smooth) dynamical system $(A,\mathbb{T}^n,\alpha)$ is a (smooth) noncommutative principal $\mathbb{T}^n$-bundle if and only if there exists a family of elements $(z_i)_{i\in I}\subseteq Z$ satisfying the following two conditions:
\begin{itemize}
\item[(i)]
The family $(D(z_i))_{i\in I}$ is an open covering of $\Gamma_Z$.
\item[(ii)]
The localized dynamical systems $(A_{\{z_i\}},\mathbb{T}^n,\alpha_{\{z_i\}})$ are (smooth) trivial noncommutative principal $\mathbb{T}^n$-bundles.
\end{itemize}
\end{remark}


\begin{lemma}\label{T^n action on spec III again}
Let $A$ be a commutative unital locally convex algebra and $(A,\mathbb{T}^n,\alpha)$ be a noncommutative principal $\mathbb{T}^n$-bundle. Further, let $\chi\in\Gamma_Z$ and $z\in Z$ with $\chi(z)\neq 0$ such that the corresponding \emph{(}smooth\emph{)} localized dynamical system $(A_{\{z\}},\mathbb{T}^n,\alpha_{\{z\}})$ is a trivial noncommutative principal $\mathbb{T}^n$-bundle. Then the $\mathbb{T}^n$-action map 
\[\sigma_{\{z\}}:\Gamma^{\emph{cont}}_{A_{\{z\}}}\times \mathbb{T}^n\rightarrow\Gamma^{\emph{cont}}_{A_{\{z\}}},\,\,\,(\chi,t)\mapsto \chi\circ\alpha_{\{z\}}(t)
\]of Lemma \ref{T^n action on spec II again} is free.
\end{lemma}

\begin{proof}
The assertion follows from fact that each trivial noncommutative principal torus bundle defines a free action of $\mathbb{T}^n$ on its spectrum (cf. \cite[Proposition 1.4]{Wa11b}).
\end{proof}

\section*{Reconstruction of Principal Torus Bundles}

In this part of the section we show how to recover the classical definition of principal torus bundles. We start with some general result relating smooth group actions and the corresponding smooth dynamical systems.

\begin{lemma}\label{extending lemma for dynamical systems}
Let $P$ be a compact manifold, $G$ a Lie group and $(C^{\infty}(P),G,\alpha)$ a smooth dynamical system. Then each character $\chi:Z\rightarrow\mathbb{C}$ is the restriction of an evaluation homomorphism 
\[\delta_p:C^{\infty}(P)\rightarrow\mathbb{C},\,\,\,f\mapsto f(p)\,\,\,\text{for some}\,\,\,p\in P.
\]
\end{lemma}

\begin{proof}
Let $\chi:Z\rightarrow\mathbb{C}$ be a character. 
If all functions of $I:=\ker\chi$ vanish in the point $p$ in $P$, then the maximality of $I$ implies that $I=(\ker\delta_p)\cap Z$, i.e., that $\chi={\delta_p}_{\mid Z}$. So we have to show that such a point exists. Let us assume that this is not the case. From that we shall derive the contradiction $I=Z$:

(i) If such a point does not exist, then for each $p\in P$ there exists a function $f_p\in I$ with $f_p(p)\neq 0$. The family $(f_p^{-1}(\mathbb{C}^{\times}))_{p\in P}$ is an open cover of $P$, so that there exist $p_1,\ldots,p_n\in P$ and a smooth $G$-invariant function $f_P:=\sum_{i=1}^n \vert f_{p_i}\vert^2>0$ on $P$. Here, the last statement follows from the fact that each algebra automorphism of $C^{\infty}(P)$ is automatically a *-automorphism (cf. \cite[Remark 1.8]{Wa11b}).

(ii) In view of part (i) we obtain a function $f_P\in I\cap Z$ which is nowhere zero. Thus, $f_P$ is invertible in $C^{\infty}(P)$. We claim that its inverse $h_P:=f_P^{-1}=\frac{1}{f_P}$ is also $G$-invariant: For this we choose an arbitrary element $g\in G$ and note that $f_P\cdot h_P=1$ implies
\[\alpha(g)(h_P)\cdot f_P=\alpha(g)(h_P\cdot f_P)=1.
\]In particular, $\alpha(g)(h_P)=h_P$. Since $g\in G$ was arbitrary, we conclude that $h_P\in Z$.

(iii) Finally, part (ii) leads to the contradiction $f_P\in I\cap Z^{\times}$, i.e., $I=Z$ as desired.
\end{proof}

\begin{proposition}\label{extension of char on fixed point algebras VI}
Let $P$ be a compact manifold, $G$ a compact Lie group and $(C^{\infty}(P),G,\alpha)$ a smooth dynamical system. Then the map
\[\Phi:P/G\rightarrow\Gamma_Z,\,\,\,p.G\mapsto \delta_p
\]is a homeomorphism.
\end{proposition}

\begin{proof}
We divide the proof of this proposition into four parts:

(i) We first note that the quotient $P/G$ is a compact Hausdorff space. Further, the map $\Phi$ is well-defined since each function $f\in Z$ is $G$-invariant. 

(ii) Proposition \ref{spec of C(M) top} and the definition of the quotient topology on $P/G$ implies that $\Phi$ is continuous.

(iii) The surjectivity of $\Phi$ is a consequence of Lemma \ref{extending lemma for dynamical systems} (alternatively we can use \cite[Corollary 2.5]{Wa11d}). To show that $\Phi$ is injective, we choose elements $p,p'\in P$ with $p.G\neq p'.G$. Since $P$ is a manifold, there exist disjoint $G$-invariant open subsets $U$ and $V$ of $P$ containing the compact subsets $p.G$ and $p'.G$, respectively. Next, we choose a positive smooth function $f$ on $P$ which is equal to $1$ on $p.G$ and vanishes outside $U$. Then the function 
\[F:P\rightarrow\mathbb{C},\,\,\,F(q):=\int_Gf(q.g)\,dg,
\]where $dg$ denotes the normalized Haar measure on $G$, defines a smooth $G$-invariant function satisfying 
\[F(p)=1\neq 0=F(p').
\]Here, we have used that the smooth action map $\alpha$ induces a smooth action of $G$ on $P$ (cf. \cite[Proposition 2.7]{Wa11c}). Hence,
\[\delta_p(F)=F(p)\neq F(p')=\delta_{p'}(F)
\]implies that $\delta_p\neq\delta_{p'}$, i.e., $\Phi$ is injective. 

(iv) So far we have seen that $\Phi$ is a bijective continuous map from a compact Hausdorff space onto a Hausdorff space. Thus, the claim follows from a well-known theorem from topology guaranteeing that $\Phi$ is a homeomorphism.
\end{proof}

For a compact manifold $P$, a compact Lie group $G$ and a dynamical system $(C^{\infty}(P),G,\alpha)$ we write $\pr:P\rightarrow P/G$, $p\mapsto p.G$ for the canonical quotient map. Moreover, we identify $D(z)$ with an open subset of $P/G$ via the homeomorphism of Proposition \ref{extension of char on fixed point algebras VI} and define
\[P_{D(z)}:=\pr^{-1}(D(z)).
\]For the following lemma we recall that $D_A(z):=\{\chi\in\Gamma_A:\,\chi(z)\neq 0\}$:

\begin{lemma}\label{C(P) localized II again}
Let $P$ be a compact manifold, $G$ a compact Lie group and $(C^{\infty}(P),G,\alpha)$ a dynamical system. If $A=C^{\infty}(P)$ and $z\in Z$, then the map
\[\Phi_z:P_{D(z)}\rightarrow D_A(z),\,\,\,p\mapsto\delta_p
\]is a homeomorphism.
\end{lemma}

\begin{proof}
This lemma is an easy consequence of Proposition \ref{spec of C(M) top}.
\end{proof}

\begin{proposition}\label{C(P) localized III again}
Suppose we are in the situation of Lemma \ref{C(P) localized II again}. Then the map 
\[\Psi_{\{z\}}:P_{D(z)}\rightarrow\Gamma^{\emph{cont}}_{C^{\infty}(P)_{\{z\}}},\,\,\,\Psi_{\{z\}}(p)([F]):=F\left(\frac{1}{z(p)},p\right)
\]is a homeomorphism. In particular, $\Gamma^{\emph{cont}}_{C^{\infty}(P)_{\{z\}}}$ becomes a smooth manifold when endowed with the smooth structure turning the map $\Psi_{\{z\}}$ into a diffeomorphism.
\end{proposition}

\begin{proof}
The claim is a direct consequence of Remark \ref{spec A_z}) and Lemma \ref{C(P) localized II again}. Indeed, we just have to note that $\Psi_{\{z\}}=\Phi_{\{z\}}\circ\Phi_z$.. 
\end{proof}

\begin{remark}\label{free=locally free}
(Freeness is a local condition). Let $\sigma: X\times G\rightarrow X$ be an action of a topological group $G$ on a topological space $X$ and let $\pr:X\rightarrow X/G$ be the corresponding quotient map. Then a short observation shows that the action $\sigma$ is free if and only if there exists an open cover $(U_i)_{i\in I}$ of the quotient space $X/G$ (with respect to the quotient topology) such that each restriction map
\[\sigma_i:X_{U_i}\times G\rightarrow X_{U_i},\,\,\,\sigma_i(x,g):=\sigma(x,g),
\]where $X_{U_i}:=\pr^{-1}(U_i)$, is free.
\end{remark}

\begin{theorem}\label{NCT^nB for manifold again}
\emph{(}Reconstruction Theorem\emph{)}. For a manifold $P$, the following assertions hold:
\begin{itemize}
\item[(a)]
If $P$ is compact and $(C^{\infty}(P),\mathbb{T}^n,\alpha)$ is a smooth noncommutative principal $\mathbb{T}^n$-bundle, then we obtain a principal $\mathbb{T}^n$-bundle $(P,P/\mathbb{T}^n,\mathbb{T}^n,\pr,\sigma)$, where
\begin{align*}
\sigma:P\times\mathbb{T}^n\rightarrow P,\,\,\,(\delta_p,t)\mapsto\delta_p\circ\alpha(t).\notag
\end{align*}
Here, we have identified $P$ with the set of characters via the map $\Phi$ from Proposition \ref{spec of C(M) top}.
\item[(b)]
Conversely, if $(P,M,\mathbb{T}^n,q,\sigma)$ is a principal $\mathbb{T}^n$-bundle, then the corresponding smooth dynamical system $(C^{\infty}(P),\mathbb{T}^n,\alpha)$ is a smooth noncommutative principal $\mathbb{T}^n$-bundle.
\end{itemize}
\end{theorem}

\begin{proof}
(a) We first note that the induced action map $\sigma$ is actually smooth (cf. \cite[Proposition 2.7]{Wa11c}). Since $\mathbb{T}^n$ is compact, $\sigma$ is proper and in view of the Quotient Theorem (cf. \cite[Kapitel VIII, Abschnitt 21]{To00}), it remains to verify the freeness of $\sigma$: For this, we first note that the map $\sigma$ is free if and only if there exists an open cover of $\Gamma_Z$ of the form $(D(z_i))_{i\in I}$ such that each restriction map
\[\sigma_i:P_{D(z_i)}\times\mathbb{T}^n\rightarrow P_{D(z_i)},\,\,\,\sigma_i(\delta_p,g):=\sigma(\delta_p,g)
\]is free (cf. Remark \ref{free=locally free}). Moreover, Proposition \ref{C(P) localized III again} implies that
\[\Psi_{\{z\}}(\sigma(\delta_p,t))=\Psi_{\{z\}}(\delta_p\circ\alpha(t))=\Psi_{\{z\}}(\delta_p)\circ\alpha_{\{z\}}(t)=\sigma_{\{z\}}(\Psi_{\{z\}}(\delta_p),t)
\]holds for each $z\in Z$, $p\in P$ and $t\in\mathbb{T}^n$. Therefore, the freeness of $\sigma$ is a consequence of Remark \ref{remark on NCPT^nB again} (i) and Lemma \ref{T^n action on spec III again}.

(b) Conversely, we have to show that the condition of Definition \ref{NCPT^nB again} is satisfied: This will be done in the following three steps:

(i) We first note that $Z\cong C^{\infty}(M)$ (cf. \cite[Proposition 2.4]{Wa11c}). Hence, $\Gamma_Z$ is homeomorphic to $M$ by Proposition \ref{spec of C(M) top}.

(ii) Next, we choose an open cover $(U_i)_{i\in I}$ of $M$ such that each $P_{U_i}$ is a trivial principal $G$-bundle over $U_i$, i.e., such that $P_{U_i}\cong U_i\times\mathbb{T}^n$ holds for each $i\in I$. Further, we choose for all $i\in I$ a $U_i$-defining function $f_i$ and note that each function $h_i:=f_i\circ q$ is a (smooth) $P_{U_i}$-defining function.

(iii) For $p\in P$ we choose $i\in I$ with $q(p)\in U_i$. Then $h_i$ is an element in $Z$ with $h_i(p)\neq 0$ and we conclude from Corollary \ref{C(M)_f=C(M_f)} that the map
\[\phi_i:C^{\infty}(P)_{\{h_i\}}\rightarrow C^{\infty}(P_{U_i}),\,\,\,[F]\mapsto\left(p\mapsto F\left(\frac{1}{h_i(p)},p\right)\right)
\]is an isomorphism of unital Fr\'{e}chet algebras. Moreover, the map $\phi_i$ is $\mathbb{T}^n$-equivariant. In fact, we have
\[(\phi_i\circ\alpha_{\{h_i\}})(g,[F])=\alpha(g,\phi_i([F]))
\]for all $t\in\mathbb{T}^n$ and $[F]\in C^{\infty}(P)_{\{h_i\}}$. In particular, the corresponding localized dynamical system $(C^{\infty}(P)_{\{h_i\}},\mathbb{T}^n,\alpha_{\{h_i\}})$ carries the structure of a smooth trivial noncommutative principal $\mathbb{T}^n$-bundle, since it is isomorphic to the smooth trivial noncommutative principal $\mathbb{T}^n$-bundle $(C^{\infty}(P_{U_j}),\mathbb{T}^n,\alpha_{U_i})$ corresponding to the trivial principal $\mathbb{T}^n$-bundle $P_{U_j}$ over $U_i$.
\end{proof}


\section*{Example 1: Trivial Noncommutative Principal Bundles}\index{Trivial NCP!Bundles}

We show that each trivial noncommutative principal $\mathbb{T}^n$-bundle carries the structure of a noncommutative principal $\mathbb{T}^n$-bundle:

\begin{theorem}\label{TNCTB are NCTB}
\emph{(}Trivial noncommutative principal torus bundles\emph{)}. Each trivial noncommutative principal $\mathbb{T}^n$-bundle $(A,\mathbb{T}^n,\alpha)$ carries the structure of a noncommutative principal $\mathbb{T}^n$-bundle.
\end{theorem}

\begin{proof}
If $(A,\mathbb{T}^n,\alpha)$ is a trivial noncommutative principal $\mathbb{T}^n$-bundle, then $1_A\in Z$ and $\chi(1_A)=1\neq 0$ holds for each $\chi\in\Gamma_Z$. In particular, Corollary \ref{inverting 1 iso} implies that $A_{\{1_A\}}\cong A$ holds as unital locally convex algebras. Therefore, $(A,\mathbb{T}^n,\alpha)$ is a noncommutative principal $\mathbb{T}^n$-bundle in its own right.
\end{proof}

\section*{Example 2: Sections of Algebra Bundles with Trivial\\Noncommutative Principal Torus Bundles as Fibre}\label{examples of NCP T^n-bundles II}

We show that if $A$ is a unital Fr\'echet algebra and $(A,\mathbb{T}^n,\alpha)$ a smooth trivial noncommutative principal $\mathbb{T}^n$-bundle such that $Z=C_A^{\mathbb{T}^n}$ is isomorphic to $\mathbb{C}$, then the algebra of sections of each algebra bundle with ``fibre" $(A,\mathbb{T}^n,\alpha)$ is a noncommutative principal $\mathbb{T}^n$-bundle. We start with the following proposition:

\begin{proposition}\label{trivial NCP T^n-bundles from trivial algebra bundles}
If $(A,\mathbb{T}^n,\alpha)$ is a \emph{(}smooth\emph{)} trivial noncommutative principal $\mathbb{T}^n$-bundle and $M$ a manifold, then the triple $(C^{\infty}(M,A),\mathbb{T}^n,\beta)$ of Lemma \ref{dynamical system and trivial algebra bundles} is also a \emph{(}smooth\emph{)} trivial noncommutative principal $\mathbb{T}^n$-bundle.
\end{proposition}

\begin{proof}
The claim directly follows from the fact that the algebra $A$ is naturally embedded in $C^{\infty}(M,A)$ through the constant maps. In fact, if ${\bf k}\in\mathbb{Z}^n$, then the corresponding isotypic component $C^{\infty}(M,A)_{\bf k}$ is equal to $C^{\infty}(M,A_{\bf k})$ and thus contains invertible elements.
\end{proof}

\begin{definition}\label{automorphism of dynamical systems}
(The automorphism group of a dynamical system). Let $(A,G,\alpha)$ be a dynamical system. The group
\[\Aut_G(A):=\{\varphi\in\Aut(A):\,(\forall g\in G)\,\alpha(g)\circ\varphi=\varphi\circ\alpha(g)\}
\]is called the \emph{automorphism group of} the dynamical system $(A,G,\alpha)$.
\end{definition}

\begin{example}
If $A=C^{\infty}(M\times G)$, then [Ne08b], Proposition 1.4.8 implies that 
\[\Aut_G(A)=C^{\infty}(M,G)\rtimes_{\gamma}\Diff(M),
\]where $\gamma:\Diff(M)\rightarrow\Aut(C^{\infty}(M,G))$, $\gamma(\varphi).f:=f\circ\varphi^{-1}$. In particular, if $A=C^{\infty}(G)$, then 
\[\Aut_G(A)\cong G.
\]
\end{example}

\begin{example}\label{G-automorphism group of quantumtori}
If $\mathbb{T}^n_{\theta}$ is the smooth noncommutative $n$-torus and $(\mathbb{T}^n_{\theta},\mathbb{T}^n,\alpha)$ the corresponding smooth dynamical system of Example \ref{NC n-tori as NCPTB} (b), then
\[\Aut_{\mathbb{T}^n}(\mathbb{T}^n_{\theta})\cong \mathbb{T}^n.
\]Indeed, we first recall that $\mathbb{T}^n_{\theta}$ is generated by unitaries $U_1,\ldots, U_n$. If now $\varphi\in\Aut_{\mathbb{T}^n}(\mathbb{T}^n_{\theta})$, $z\in\mathbb{T}^n$ and $U_r$ ($1\leq r\leq n$) is such a 
unitary, then 
\[z.\varphi(U_r)=\varphi(z.U_r)=\varphi(z_r\cdot U_r)=z_r\cdot\varphi(U_r)
\]leads to $\varphi(U_r)=\lambda_r\cdot U_r$ for some $\lambda_r\in\mathbb{C}^{\times}$. Since $U_r$ is a unitary and $\varphi$ is a $^*$-automorphism ($\mathbb{T}^n_{\theta}$ is a $^*$-algebra), we conclude that $\lambda_r\in\mathbb{T}$. In particular, each element in $\Aut_{\mathbb{T}^n}(\mathbb{T}^n_{\theta})$ corresponds to an element in $\mathbb{T}^n$ and vice versa. 
\end{example}

\begin{proposition}\label{cocycle description of bundles}
\emph{(}Cocycle description of bundles\emph{)}. Let $(A,G,\alpha)$ be a smooth dynamical system and $M$ a manifold. Further, let $(U_i)_{i \in I}$ be an open cover of $M$ and $U_{ij}:=U_i\cap U_j$ for $i,j\in I$. If $(g_{ij})_{i,j \in I}$ is a collection of functions $g_{ij}\in C^{\infty}(U_{ij},\Aut_G(A))$ satisfying 
\[g_{ii}={\bf 1}\,\,\,\text{and}\,\,\,g_{ij}g_{jk}=g_{ik}\,\,\,\text{on}\,\,\,U_{ijk}:=U_i\cap U_j\cap U_k,
\]then the following assertions hold:
\begin{itemize}
\item[(a)]
There exists an algebra bundle $(\mathbb{A},M,A,q)$ and bundle charts $\varphi_i:U_i\times A\rightarrow\mathbb{A}_{U_i}$ such that
\[(\varphi_i^{-1}\circ\varphi_j)(x,a)=(x,g_{ij}(x).a).
\]
\item[(b)]
The map
\[\sigma:G\times\mathbb{A}\rightarrow\mathbb{A},\,\,\,(g,a)\mapsto\varphi_i(x,\alpha(g).a_0),
\]for $i\in I$ with $x=q(a)\in U_i$ and $a_0\in A$ with $\varphi_i(x,a_0)=a$, defines a smooth action of $G$ on $\mathbb{A}$ by fibrewise algebra automorphisms.
\end{itemize}
\end{proposition}

\begin{proof}
(a) A proof of the first statement can be found in \cite[Part I, Section 3, Theorem 3.2]{St51}.

(b) The crucial point is to show that the map $\sigma$ is well-defined: For this let $i,j\in I$ with $x=q(a)\in U_{ij}$, $a_0\in A$ with $\varphi_i(x,a_0)=a$ and $a'_0\in A$ with $\varphi_j(x,a'_0)=a$. Then $a_0=g_{ij}(x).a'_0$ leads to
\begin{align}
\varphi_i(x,\alpha(g).a_0)&=\varphi_i(x,\alpha(g)(g_{ij}(x).a'_0))=\varphi_i(x,g_{ij}(x)(\alpha(g).a'_0))=\varphi_j(x,\alpha(g).a'_0).\notag
\end{align}
Further, a short calculation shows that $\sigma$ defines an action of $G$ on $\mathbb{A}$ by fibrewise algebra automorphisms. Its smoothness follows from the local description by a smooth function.
\end{proof}

\begin{proposition}\label{dynamical system and algebra bundles}
Suppose we are in the situation of Proposition \ref{cocycle description of bundles}. Then the map
\[\beta:G\times\Gamma\mathbb{A}\rightarrow\Gamma\mathbb{A},\,\,\,\beta(g,s)(m):=\sigma(g,s(m))
\]defines a smooth action of $G$ on $\Gamma\mathbb{A}$ by algebra automorphisms. In particular, the triple $(\Gamma\mathbb{A},G,\beta)$ is a smooth dynamical system.
\end{proposition}

\begin{proof}
That the map $\beta$ defines an action of $G$ on $\Gamma\mathbb{A}$ by algebra automorphisms is a consequence of Proposition \ref{cocycle description of bundles} (b). To verify the smoothness of $\beta$, we first choose a bundle atlas $(\varphi_i,U_i)_{i\in I}$ of $(\mathbb{A},M,A,q)$ and use the definition of the smooth structure on $\Gamma\mathbb{A}$: In fact, Remark \ref{smooth structure on sections} implies that the map $\beta$ is smooth if and only if each map 
\[\Phi_i\circ\beta:G\times\Gamma\mathbb{A}\rightarrow C^{\infty}(U_i,A),\,\,\,(g,s)\mapsto\alpha(g)\circ s_i
\]is smooth. Next, we recall that each map
\[\alpha_i:G\times C^{\infty}(U_i,A)\rightarrow C^{\infty}(U_i,A),\,\,\,(g,f)\mapsto\alpha(g)\circ f
\]is smooth by Lemma \ref{dynamical system and trivial algebra bundles}. Since $\Phi_i\circ\beta=\alpha_i\circ(\id_G\times\Phi_i)$ for each $i\in I$, each map $\Phi_i\circ\beta$ is smooth as a composition of smooth maps.
\end{proof}


\begin{remark}
For the next theorem we recall that the smooth dynamical systems $(\mathbb{T}^n_{\theta},\mathbb{T}^n,\alpha)$ of Example \ref{NC n-tori as NCPTB} (b) provide a class of examples of trivial noncommutative principal $\mathbb{T}^n$-bundles for which $C_{\mathbb{T}^n_{\theta}}^{\mathbb{T}^n}$ is isomorphic to $\mathbb{C}$. 
\end{remark}

\begin{theorem}\label{NCP G_bundle with fibre trivial NCP G bundle}
Let $A$ be a unital Fr\'echet algebra and $(A,\mathbb{T}^n,\alpha)$ a smooth trivial noncommutative principal $\mathbb{T}^n$-bundle such that $C_A^{\mathbb{T}^n}$ is isomorphic to $\mathbb{C}$. Further, let $M$ be a manifold, $(U_i)_{i \in I}$ an open cover of $M$ and $U_{ij}:=U_i\cap U_j$ for $i,j\in I$. If $(g_{ij})_{i,j \in I}$ is a collection of functions $g_{ij}\in C^{\infty}(U_{ij},\Aut_{\mathbb{T}^n}(A))$ satisfying 
\[g_{ii}={\bf 1}\,\,\,\text{and}\,\,\,g_{ij}g_{jk}=g_{ik}\,\,\,\text{on}\,\,\,U_{ijk}:=U_i\cap U_j\cap U_k,
\]then the smooth dynamical system $(\Gamma\mathbb{A},\mathbb{T}^n,\beta)$ of Proposition \ref{dynamical system and algebra bundles} is a noncommutative principal $\mathbb{T}^n$-bundle.
\end{theorem}

\begin{proof}
(i) We first note that $Z:=C_{\Gamma\mathbb{A}}^{\mathbb{T}^n}\cong C^{\infty}(M)$. In particular, the spectrum $\Gamma_Z$ is homeomorphic to $M$.

(ii) Next, we choose for each $i\in I$ a $U_i$-defining function $f_i$. Further, for $m\in M$, we choose $i\in I$ with $m\in U_i$. Then $f_i$ is an element in $C^{\infty}(M)$ with $f_i(m)\neq 0$ and we conclude from Corollary \ref{algebra section 3} that the map
\[\phi_{U_i}:\Gamma\mathbb{A}_{\{f_i\}}\rightarrow\Gamma\mathbb{A}_{U_i},\,\,\,[F]\mapsto F\circ\left(\frac{1}{f_i}\times\id_{U_i}\right)
\]is an isomorphism of unital Fr\'echet algebras. Moreover, the map $\phi_{U_i}$ is $\mathbb{T}^n$-equivariant. In fact, we have
\[(\phi_{U_i}\circ\beta_{\{f_i\}})(t,[F])=\beta(t,\phi_{U_i}([F]))
\]for all $t\in\mathbb{T}^n$ and $[F]\in\Gamma\mathbb{A}_{\{f_i\}}$. Since the natural isomorphism between the space $\Gamma\mathbb{A}_{U_i}$ and $C^{\infty}(U_i,A)$ is also a $\mathbb{T}^n$-equivariant isomorphism of unital Fr\'echet algebras (cf. Proposition \ref{cocycle description of bundles} (a)), it follows from Proposition \ref{trivial NCP T^n-bundles from trivial algebra bundles} that the corresponding dynamical system $(\Gamma\mathbb{A}_{\{f_i\}},\mathbb{T}^n,\beta_{\{f_i\}})$ of Proposition \ref{T^n action on spec again} carries the structure of a smooth trivial noncommutative principal $\mathbb{T}^n$-bundle.
\end{proof}

\begin{example}\label{non-triviality of the previous construction}
(Non-triviality of the previous construction). In this example we show that the previous construction actually leads to non-trivial examples. For this we apply Theorem \ref{NCP G_bundle with fibre trivial NCP G bundle} to the trivial noncommutative principal $\mathbb{T}^n$-bundle $(\mathbb{T}^n_{\theta},\mathbb{T}^n,\alpha)$. In view of Example \ref{G-automorphism group of quantumtori} we have 
\begin{align}
\Aut_{\mathbb{T}^n}(\mathbb{T}^n_{\theta})\cong \mathbb{T}^n.\label{referenz G-automorphism group of quantumtori}
\end{align}
In particular, a similar argument as in \cite[Proposition 2.1.14]{Wa11a} implies that there is a one-to-one correspondence between the algebra bundles arising from Proposition \ref{cocycle description of bundles} and principal $\mathbb{T}^n$-bundles. Thus, if $(\mathbb{A},M,\mathbb{T}^n_{\theta},q)$ is such an algebra bundle which corresponds to a non-trivial principal $\mathbb{T}^n$-bundle, then also $(\mathbb{A},M,\mathbb{T}^n_{\theta},q)$ is non-trivial as algebra bundle. We claim that the associated smooth dynamical system $(\Gamma\mathbb{A},\mathbb{T}^n,\beta)$ of Proposition \ref{dynamical system and algebra bundles} is a non-trivial noncommutative principal $\mathbb{T}^n$-bundle. To prove this claim we assume the converse, i.e., that $(\Gamma\mathbb{A},\mathbb{T}^n,\beta)$ is a trivial noncommutative principal $\mathbb{T}^n$-bundle. For this, we rename the open subsets $U_i$ of $M$ of Theorem \ref{NCP G_bundle with fibre trivial NCP G bundle} to $O_i$ in order to avoid an abuse of notation and proceed as follows:

(i) We first recall that $\mathbb{T}^n_{\theta}$ is generated by unitaries $U_1,\ldots,U_n$ and that its elements are given by (norm-convergent) sums
\[a=\sum_{{\bf k}\in\mathbb{Z}^n}a_{\bf k}U^{\bf k},\,\,\,\text{with}\,\,\,(a_{\bf k})_{{\bf k}\in\mathbb{Z}^n}\in S(\mathbb{Z}^n).
\]Here, 
\[U^{\bf k}:=U^{k_1}_1\cdots U^{k_n}_n.
\]

(ii) For each $1\leq r\leq n$ let $(\Gamma\mathbb{A})_r$ be the isotypic component corresponding to the canonical basis element $e_r=(0,\ldots,1,\ldots,0)$ of $\mathbb{Z}^n$. Then the definition of the action $\beta$ implies that
\begin{align}
(\Gamma\mathbb{A})_r&=\{s\in\Gamma\mathbb{A}:\,(\forall z\in\mathbb{T}^n)\,z.s=z_r\cdot s\}\notag\\
&=\{s\in\Gamma\mathbb{A}:\,(\forall m\in M)\,s(m)\in(\mathbb{T}^n_{\theta})_{m,r}\},\notag
\end{align}
where $(\mathbb{T}^n_{\theta})_{m,r}$ denotes the isotypic component of the fibre $(\mathbb{T}^n_{\theta})_m$ corresponding to $e_r$.
Since $(\Gamma\mathbb{A},\mathbb{T}^n,\beta)$ is assumed to be a trivial noncommutative principal $\mathbb{T}^n$-bundle, we may choose in each space $(\Gamma\mathbb{A})_r$ an invertible element $s_r:M\rightarrow\Gamma\mathbb{A}$.

(iii) If $s_{r,i}:=\pr_{\mathbb{T}^n_{\theta}}\circ\varphi_i^{-1}\circ {s_r}_{\mid O_i}:O_i\rightarrow \mathbb{T}^n_{\theta}$ (cf. Construction \ref{top on space of sections}), then $s_{r,i}=\lambda_{r,i}\cdot U_r$ for some smooth function $\lambda_{r,i}:O_i\rightarrow\mathbb{C}^{\times}$. Moreover, the compatibility property for $(s_{r,i})_{i\in I}$ and (\ref{referenz G-automorphism group of quantumtori}) imply that $\vert\lambda_{r,j}\vert=\vert\lambda_{r,i}\vert$ for all $i,j\in I$. Hence, we may choose a section $s_r:M\rightarrow\Gamma\mathbb{A}$ which is locally given by $\lambda_{r,i}\cdot U_r$ for some smooth function $\lambda_{r,i}:O_i\rightarrow\mathbb{T}$.

(iv) We now show that the map
\[\varphi:M\times\mathbb{T}^n_{\theta}\rightarrow\mathbb{A},\,\,\,\left(m,a=\sum_{{\bf k}\in\mathbb{Z}^n}a_{\bf k}U^{\bf k}\right)\mapsto\sum_{{\bf k}\in\mathbb{Z}^n}a_{\bf k}s_1(m)^{k_1}\cdots s_n(m)^{k_n}
\]is an equivalence of algebra bundles over $M$. Indeed, $\varphi$ is bijective and fibrewise an algebra automorphism. Moreover, the map $\varphi$ is smooth if and only if the map 
\[\psi_i:=\pr_{\mathbb{T}^n_{\theta}}\circ\varphi_i^{-1}\circ\varphi_{\mid O_i\times \mathbb{T}^n_{\theta}}:O_i\times \mathbb{T}^n_{\theta}\rightarrow\mathbb{T}^n_{\theta}, \left(x,a=\sum_{{\bf k}\in\mathbb{Z}^n}a_{\bf k}U^{\bf k}\right)\mapsto\sum_{{\bf k}\in\mathbb{Z}^n}a_{\bf k}s_{1,i}(x)^{k_1}\cdots s_{n,i}(x)^{k_n}
\]is smooth for each $i\in I$. Since 
\[\psi_i(x,a)=\sum_{{\bf k}\in\mathbb{Z}^n}a_{\bf k}\lambda_{1,i}(x)^{k_1}\cdots\lambda_{n,i}(x)^{k_n}U^{\bf k},
\]we conclude that $\psi_i=\alpha\circ((\lambda_{1,i},\ldots,\lambda_{n,i})\times\id_{\mathbb{T}^n_{\theta}})$, i.e., that $\psi_i$ is smooth as a composition of smooth maps. A similar argument shows the smoothness of the inverse map.

(v) We finally achieve the desired contradiction: In view of part (iv), $\mathbb{A}$ is a trivial algebra bundle contradicting the construction of $\mathbb{A}$, i.e., that $\mathbb{A}$ is non-trivial as algebra bundle. This proves the claim.
\end{example}

\begin{remark}\label{nontriviality of the bundle corresponding to A^2}
In this remark we want to point out that there exist non-trivial algebra bundles which are trivial as noncommutative principal torus bundles: In fact, \cite[Corollary 12.7]{GVF01} implies that the $2$-tori $A^2_{\theta}$ are mutually non-isomorphic for $0\leq\theta\leq\frac{1}{2}$. Further, if $\theta$ is rational, $\theta=\frac{n}{m}$, $n\in\mathbb{Z}$, $m\in\mathbb{N}$ relatively prime, then $A^2_{\theta}$ is isomorphic to the algebra of continuous sections of an algebra bundle over $\mathbb{T}^2$ with fibre $M_m(\mathbb{C})$ (cf. \cite[Proposition 12.2]{GVF01} or \cite{Wa11a}, Proposition E.2.5 for the smooth case). Therefore, the (continuous) bundle corresponding to such a rational quantum torus $A^2_{\theta}$, $\theta$ rational with $0<\theta\leq\frac{1}{2}$, is non-trivial as algebra bundle. Since the smooth noncommutative $2$-torus $\mathbb{T}^2_{\theta}$ is a dense subalgebra of $A^2_{\theta}$, the same conclusion holds for the (smooth) bundle corresponding to $\mathbb{T}^2_{\theta}$. Nevertheless, the associated dynamical systems $(A^2_{\theta},\mathbb{T}^n,\alpha)$ and $(\mathbb{T}^2_{\theta},\mathbb{T}^n,\alpha)$ of Example \ref{NC n-tori as NCPTB} are trivial noncommutative principal $\mathbb{T}^2$-bundles.
\end{remark}

\section*{Example 3: Sections of Algebra Bundles which are Pull-Backs of Principal Torus Bundles}\label{examples of NCP T^n-bundles III}

We show that if $A$ is a unital Fr\'echet algebra with trivial center, $(\mathbb{A},M,A,q)$ an algebra bundle and $(P,M,\mathbb{T}^n,\pi,\sigma)$ a principal $\mathbb{T}^n$-bundle, then the algebra of sections of the pull-back bundle 
\[\pi^{*}(\mathbb{A}):=\{(p,a)\in P\times\mathbb{A}:\,\pi(p)=q(a)\}
\]is a noncommutative principal $\mathbb{T}^n$-bundle. We start with the following lemma:

\begin{lemma}\label{pull-back 0}
If $A$ is a unital locally convex algebra and $M$ a manifold, then the map
\[\alpha:G\times C^{\infty}(M\times\mathbb{T}^n,A)\rightarrow C^{\infty}(M\times\mathbb{T}^n,A),\,\,\,(t,f)\mapsto (t.f)(m,t'):=f(m,tt')
\]defines a smooth action of $\mathbb{T}^n$ on $C^{\infty}(M\times \mathbb{T}^n,A)$ by algebra automorphisms. In particular, the triple $(C^{\infty}(M\times \mathbb{T}^n,A),\mathbb{T}^n,\alpha)$ is a smooth trivial noncommutative principal $\mathbb{T}^n$-bundle.
\end{lemma}

\begin{proof}
For the proof we just have to note that the algebra $C^{\infty}(M\times\mathbb{T}^n)$ is naturally embedded in $C^{\infty}(M\times \mathbb{T}^n,A)$ through the unit element of $A$. The smoothness of the map $\alpha$ can be proved similarly to Lemma \ref{dynamical system and trivial algebra bundles}.
\end{proof}

\begin{lemma}\label{pull-back I}
If $(\mathbb{A},M,A,q)$ is an algebra bundle, $(P,M,G,\pi,\sigma)$ a principal bundle and $\pi^{*}(\mathbb{A})$ the pull-back bundle over $P$, then the map
\[\sigma:\pi^{*}(\mathbb{A})\times G\rightarrow\pi^{*}(\mathbb{A}),\,\,\,((p,a),g)\mapsto(p.g,a)
\]defines a smooth action of $G$ on $\pi^{*}(\mathbb{A})$.
\end{lemma}

\begin{proof}
We first note that the map $\sigma$ is well-defined. Its smoothness follows from local considerations and we leave the details to the reader.
\end{proof}

\begin{proposition}\label{pull-back II}
Suppose we are in the situation of Lemma \ref{pull-back I}. If $\mathcal{A}:=\Gamma\pi^{*}(\mathbb{A})$, then the map
\[\alpha:G\times\mathcal{A}\rightarrow\mathcal{A},\,\,\,\alpha(g,s)(p):=\sigma(s(p.g),g^{-1})
\]defines a smooth action of $G$ on $\mathcal{A}$ by algebra automorphisms. In particular, the triple $(\mathcal{A},G,\alpha)$ is a smooth dynamical system.
\end{proposition}

\begin{proof}
An easy calculation shows that the map $\alpha$ is a well-defined action of $G$ on $\mathcal{A}$ by algebra automorphisms. Next, we choose a bundle atlas $(\varphi_i,U_i)_{i\in I}$ of $(\mathbb{A},M,A,q)$ with the additional property that each $V_i:=\pi^{-1}(U_i)$ is trivial. Then $(\pi^{*}(\varphi_i),V_i)_{i\in I}$ is a bundle atlas of the pull-back bundle $\pi^{*}(\mathbb{A})$ over $P$ and we can use the definition of the smooth structure on $\mathcal{A}$ to verify the smoothness of the map $\alpha$: In fact, Remark \ref{smooth structure on sections} implies that $\alpha$ is smooth if and only if each map 
\[\Phi_i\circ\alpha:G\times\mathcal{A}\rightarrow C^{\infty}(V_i,A),\,\,\,(g,s)\mapsto s_i\circ(\sigma_g)_{\mid V_i}
\]is smooth. For this we note that each map
\[\alpha_i:G\times C^{\infty}(V_i,A)\rightarrow C^{\infty}(V_i,A),\,\,\,(g,f)\mapsto f\circ(\sigma_g)_{\mid V_i}
\]is smooth and further that $\Phi_i\circ\alpha=\alpha_i\circ(\id_G\times\Phi_i)$ holds for each $i\in I$. Thus, each map $\Phi_i\circ\beta$ is smooth as a composition of smooth maps.
\end{proof}

\begin{theorem}\label{pull-back IV}
Let $A$ be a unital Fr\'echet algebra with trivial center, $(\mathbb{A},M,A,q)$ an algebra bundle and $(P,M,\mathbb{T}^n,\pi,\sigma)$ a principal bundle. If $\pi^{*}(\mathbb{A})$ is the pull-back bundle over $P$ and $\mathcal{A}:=\Gamma\pi^{*}(\mathbb{A})$, then the smooth dynamical system $(\mathcal{A},\mathbb{T}^n,\alpha)$ of Proposition \ref{pull-back II} is a smooth noncommutative principal $\mathbb{T}^n$-bundle.
\end{theorem}

\begin{proof}
(i) We first note that $C_{\mathcal{A}}\cong C^{\infty}(P)$ and therefore that $Z=C_{\mathcal{A}}^{\mathbb{T}^n}\cong C^{\infty}(M)$ (cf. \cite[Proposition 2.4]{Wa11c}). In particular, the spectrum $\Gamma_Z$ is homeomorphic to $M$.

(ii) Next, we choose an open cover $(U_i)_{i\in I}$ of $M$ such that both $\mathbb{A}_{U_i}$ and $V_i:=P_{U_i}$ are trivial, i.e., such that $\mathbb{A}_{U_i}\cong U_i\times A$ and \mbox{$V_i\cong U_i\times\mathbb{T}^n$} hold for each $i\in I$. Further, we choose for each $i\in I$ a $U_i$-defining function $f_i$ and note that each function $h_i:=f_i\circ q$ is a (smooth) $V_i$-defining function $f_i$.

(iii) For $p\in P$ we choose $i\in I$ with $q(p)\in U_i$. Then $h_i$ is an element in $Z$ with $h_i(p)\neq 0$ and we conclude from Corollary \ref{algebra section 3} that the map
\[\phi_{V_i}:\mathcal{A}_{\{h_i\}}\rightarrow\mathcal{A}_{V_i},\,\,\,[F]\mapsto F\circ\left(\frac{1}{h_i}\times\id_{V_i}\right)
\]is an isomorphism of unital Fr\'echet algebras. Moreover, the map $\phi_{V_i}$ is $\mathbb{T}^n$-equivariant. In fact, we have
\[(\phi_{V_i}\circ\alpha_{\{h_i\}})(t,[F])=\alpha(t,\phi_{V_i}([F]))
\]for all $t\in\mathbb{T}^n$ and $[F]\in\mathcal{A}_{\{h_i\}}$. Since the natural isomorphism between the space $\mathcal{A}_{V_i}$ and $C^{\infty}(U_i\times\mathbb{T}^n,A)$ is also a $\mathbb{T}^n$-equivariant isomorphism of unital Fr\'echet algebras, the dynamical system $(\mathcal{A}_{\{h_i\}},\mathbb{T}^n,\alpha_{\{h_i\}})$ carries the structure of a smooth trivial noncommutative principal $\mathbb{T}^n$-bundle (cf. Lemma \ref{pull-back 0}).
\end{proof}

\begin{example}\label{non-triviality of the previous construction for example 3}
(Non-triviality of the previous construction). In this example we show that the previous construction actually leads to non-trivial examples. Therefore let $\SU_3(\mathbb{C})$ be the special unitary group of rank $3$. It is a well-known fact that $\SU_3(\mathbb{C})$ is a 1-connected, i.e., connected and simply-connceted, Lie group (cf. \cite[Proposition 16.1.3]{HiNe10}). Moreover, $\SU_3(\mathbb{C})$ contains a maximal torus of dimension 2. Identifying $\mathbb{T}^2$ with this torus leads to a non-trivial principal bundle $(\SU_3(\mathbb{C}),\SU_3(\mathbb{C})/\mathbb{T}^2,\mathbb{T}^2,\pr,\sigma)$, where \mbox{$\pr:\SU_3(\mathbb{C})\rightarrow\SU_3(\mathbb{C})/\mathbb{T}^2$} denotes the canonical quotient map and $\sigma$ is the natural subgroup action. Next, let us consider the trivial algebra bundle over $\SU_3(\mathbb{C})/\mathbb{T}^2$ with fibre $\M_n(\mathbb{C})$. The pull-back along $(\SU_3(\mathbb{C}),\SU_3(\mathbb{C})/\mathbb{T}^2,\mathbb{T}^2,\pr,\sigma)$ leads to the trivial algebra bundle over $\SU_3(\mathbb{C})$ with fibre $\M_n(\mathbb{C})$. We claim that the associated smooth dynamical system $(\mathcal{A},\mathbb{T}^2,\alpha)$ of Proposition \ref{pull-back II} is a non-trivial noncommutative principal $\mathbb{T}^2$-bundle. Here we have, of course, \mbox{$\mathcal{A}=C^{\infty}(\SU_3(\mathbb{C}),\M_n(\mathbb{C}))$}. To prove this claim we assume the converse, i.e., that $(\mathcal{A},\mathbb{T}^2,\alpha)$ is a trivial noncommutative principal $\mathbb{T}^2$-bundle and proceed as follows:

(i) Since $(\mathcal{A},\mathbb{T}^2,\alpha)$ is assumed to be a trivial noncommutative principal $\mathbb{T}^2$-bundle, there exist invertible elements $F\in\mathcal{A}_{(1,0)}$ and $F'\in\mathcal{A}_{(0,1)}$. 

(ii) Part (i) now implies that $f:=\Det(F)\in C^{\infty}(\SU_3(\mathbb{C}))$. Since $F$ is invertible, so is $f$, i.e., $f$ takes values in $\mathbb{C}^{\times}$. In view of \cite[Proposition 1.9]{Wa11b}, we may assume that $f$ takes values in $\mathbb{T}$. Moreover, the function $f$ satisfies
\[(t_1,t_2).f=(t_1,t_2).\Det(F)=\Det(t_1\cdot F)=t_1^n\cdot\Det(F)=t_1^n\cdot f.
\]The same construction applied to $F'$ gives an invertible element $f'\in C^{\infty}(\SU_3(\mathbb{C}))$ which takes values in $\mathbb{T}$ and satisfies
\[(t_1,t_2).f'=(t_1,t_2).\Det(F')=\Det(t_2\cdot F')=t_2^n\cdot\Det(F')=t_2^n\cdot f'.
\]Here we have used that the action of $\mathbb{T}^2$ on $\mathcal{A}$ restricts to an action on $C_{\mathcal{A}}\cong C^{\infty}(\SU_3(\mathbb{C}))$.

(iii) Let $p_n:\mathbb{T}\rightarrow\mathbb{T}$, $t\mapsto t^n$ be the $n$-fold covering map of the $1$-torus. Since $\SU_3(\mathbb{C})$ is simply connected, we can lift $f$ to a smooth function $\widetilde{f}:\SU_3(\mathbb{C})\rightarrow\mathbb{T}$ satisfying $p_n\circ\widetilde{f}=f$. Further, a short observation shows that $(t_1,t_2).\widetilde{f}=t_1\cdot\widetilde{f}$. In particular, we obtain an invertible element $\widetilde{f}\in C^{\infty}(\SU_3(\mathbb{C}))_{(1,0)}$ which takes values in $\mathbb{T}$. A similar construction leads to an invertible element $\widetilde{f'}\in C^{\infty}(\SU_3(\mathbb{C}))_{(0,1)}$ which takes values in $\mathbb{T}$. 

(iv) Finally, part (iii) leads to the desired contradiction: Indeed, we conclude that the smooth functions $\widetilde{f}$ and $\widetilde{f'}$ define an equivalence of principal $\mathbb{T}^2$-bundles over $\SU_3(\mathbb{C})/\mathbb{T}^2$
given through the map
\[\varphi:\SU_3(\mathbb{C})\rightarrow \SU_3(\mathbb{C})/\mathbb{T}^2\times\mathbb{T}^2/,\,\,\,M\mapsto(\pr(M),\widetilde{f}(M),\widetilde{f'}(M)).
\]This is not possible, since $\SU_3(\mathbb{C})/\mathbb{T}^2$ is simply-connected.
\end{example}

\begin{remark}
(a) For a discussion on infinitesimal objects for the noncommutative space $C^{\infty}(\SU_3(\mathbb{C}),\M_n(\mathbb{C}))$ we refer to [Mas08], Section 3 or \cite[Section 11.3]{Wa11a}. 

(b) The same argument as in Example \ref{non-triviality of the previous construction for example 3} works for every principal bundle $(P,M,\mathbb{T}^n,q,\sigma)$ which has a simply connected total space $P$.
\end{remark}

\section*{Example 4: Sections of Trivial Equivariant Algebra Bundles}\label{examples of NCP T^n-bundles V}

We now consider again a principal bundle $(P,M,G,q,\sigma)$ and, in addition, a unital locally convex algebra $A$. If $\pi:G\times A\rightarrow A$ defines a smooth action of $G$ on $A$ by algebra automorphisms, then 
\[(p,a).g:=(p.g,\pi(g^{-1}).a)
\]defines a (free) action of $G$ on $P\times A$ and one easily verifies that the trivial algebra bundle $(P\times A,P,A,q_P)$ is $G$-equivariant. Moreover, a short observation shows that the map
\[\alpha:G\times C^{\infty}(P,A)\rightarrow C^{\infty}(P,A),\,\,\,(g.f)(p):=\pi(g).f(p.g)
\]defines a smooth action of $G$ on $C^{\infty}(P,A)$ by algebra automorphisms. We recall that the corresponding fixed point algebra 
is isomorphic (as $C^{\infty}(M)$-module) to the space of sections of the associated algebra bundle
\[\mathbb{A}:=P\times_{\pi}A:=P\times_GA:=(P\times A)/G
\]over $M$. In fact, we refer to \cite[Construction 2.1.13]{Wa11a} if $A$ is finite-dimensional. If $A$ is not finite-dimensional, then one can use that bundle charts for $(P,M,G,q,\sigma)$ induces bundle charts for the associated algebra bundle.

\begin{lemma}\label{V.1}
The above situation applied to the trivial principal bundle $(M\times\mathbb{T}^n,M,\mathbb{T}^n,q_M,\sigma_{\mathbb{T}^n})$ leads to a smooth trivial noncommutative principal $\mathbb{T}^n$-bundle $(C^{\infty}(M\times\mathbb{T}^n,A),\mathbb{T}^n,\alpha)$ with fixed point algebra $C^{\infty}(M,A)$.
\end{lemma}

\begin{proof}
For the proof we again just have to note that the algebra $C^{\infty}(M\times\mathbb{T}^n)$ is naturally embedded in $C^{\infty}(M\times\mathbb{T}^n,A)$ through the unit element of $A$. The smoothness of the map $\alpha$ can be proved similarly to Lemma \ref{dynamical system and trivial algebra bundles}.
\end{proof}

\begin{theorem}\label{V.2}
If $(P,M,\mathbb{T}^n,\pi,\sigma)$ is a principal bundle, $A$ a unital Fr\'echet algebra with trivial center and $\pi:G\times A\rightarrow A$ a smooth action of $G$ on $A$, then the smooth dynamical system $(C^{\infty}(P,A),G,\alpha)$ is a smooth NCP $G$-bundle.
\end{theorem}

\begin{proof}
(i) We first note that $Z\cong C^{\infty}(M)$ (cf. \cite[Proposition 2.4]{Wa11c}). Hence, $\Gamma_Z$ is homeomorphic to $M$ by Proposition \ref{spec of C(M) top}.

(ii) Next, we choose an open cover $(U_i)_{i\in I}$ of $M$ such that each $P_i:=P_{U_i}$ is a trivial principal $\mathbb{T}^n$-bundle over $U_i$, i.e., such that $P_i\cong U_i\times \mathbb{T}^n$ holds for each $i\in I$. Further, we choose for all $i\in I$ a $U_i$-defining function $f_i$ and note that each function $h_i:=f_i\circ q$ is a (smooth) $P_i$-defining function $f_i$.

(iii) For $p\in P$ we choose $i\in I$ with $q(p)\in U_i$. Then $h_i$ is an element in $Z$ with $h_i(p)\neq 0$ and we conclude from Corollary \ref{algebra section 3} that the map
\[\phi_{P_i}:C^{\infty}(P,A)_{\{h_i\}}\rightarrow C^{\infty}(P_i,A),\,\,\,[F]\mapsto F\circ\left(\frac{1}{h_i}\times\id_{P_i}\right)
\]is an isomorphism of unital Fr\'echet algebras. Moreover, the map $\phi_{P_i}$ is $\mathbb{T}^n$-equivariant. In fact, we have
\[(\phi_{P_i}\circ\alpha_{\{h_i\}})(t,[F])=\alpha(t,\phi_{P_i}([F]))
\]for all $t\in\mathbb{T}^n$ and $[F]\in C^{\infty}(P,A)_{\{h_i\}}$. Since the natural isomorphism between $C^{\infty}(P_i,A)$ and $C^{\infty}(U_i\times\mathbb{T}^n,A)$ is also a $\mathbb{T}^n$-equivariant isomorphism of unital Fr\'echet algebras, the dynamical system $(C^{\infty}(P_i,A)_{\{h_i\}},\mathbb{T}^n,\alpha_{\{h_i\}})$ carries the structure of a smooth trivial noncommutative principal $\mathbb{T}^n$-bundle (cf. Lemma \ref{V.1}).
\end{proof}


\begin{remark}\label{non-triviality of the previous construction for example 4}
(Non-triviality of the previous construction). Non-trivial examples can be constructed similarly as in Example \ref{non-triviality of the previous construction for example 3}.
\end{remark}

\section*{Example 5: A Very Concrete Example of a Noncommutative Principal Torus Bundle}\label{examples of NCP T^n-bundles IV}

Let $\pi:\mathbb{R}\rightarrow\mathbb{T}$ be the universal covering of $\mathbb{T}$, $A=C^{\infty}(\mathbb{T}^2)$ and note that the map
\[\gamma:\mathbb{Z}\rightarrow\Aut(A),\,\,\,(\gamma(n).f)(z_1,z_2)=f(z_1,z_1^nz_2)
\]defines a (smooth) action of $\mathbb{Z}$ on $A$. We thus can form the associated algebra bundle
\begin{align}
q:\mathbb{A}:=\mathbb{R}\times_{\gamma} A\rightarrow\mathbb{T},\,\,\,[r,f]\mapsto\pi(r),\label{associated algebra bundle}
\end{align}
with fibre $A$.

\begin{proposition}
The space $\Gamma\mathbb{A}$ of section of the bundle \emph{(}\ref{associated algebra bundle}\emph{)} carries the structure of a non-trivial noncommutative principal bundle $\mathbb{T}$-bundle.
\end{proposition}

\begin{proof}
The proof of this claim is divided into the following five steps:

(i) We first note that the map
\[\Psi:C^{\infty}(\mathbb{R},A)^{\mathbb{Z}}\rightarrow\Gamma\mathbb{A},\,\,\,\Psi(f)(\pi(r)):=[r,f(r)],
\]is an isomorphism of unital Fr\'echet algebras. Indeed, both $C^{\infty}(\mathbb{R},A)^{\mathbb{Z}}$ and $\Gamma\mathbb{A}$ are unital Fr\'echet algebras. An easy calculation shows that the map $\Phi$ is an isomorphism of unital algebras, and its continuity is a consequence of the topology on $\Gamma\mathbb{A}$ (cf. Definition \ref{top on space of sections}). Finally, the Open Mapping Theorem (cf. \cite[Chapter III, Section 2.2]{Sch99}) implies that $\Psi$ is open.

(ii) From Lemma \ref{smooth exp law}, we conclude that
\[C^{\infty}(\mathbb{R},A)^{\mathbb{Z}}=\{f:\mathbb{R}\rightarrow A:\,(\forall r\in\mathbb{R},n\in\mathbb{Z})\,f(r+n)=\gamma(-n).f(r)\}
\]is isomorphic (as a unital Fr\'echet algebra) to
\begin{align}
\{f^{\wedge}:\mathbb{R}\times\mathbb{T}^2\rightarrow\mathbb{C}:\,(\forall r\in\mathbb{R},(z_1,z_2)\in\mathbb{T}^2,n\in\mathbb{Z})\,f^{\wedge}(r+n,z_1,z_2)=f^{\wedge}(r,z_1,z_1^{-n}z_2)\}.\label{equation example}
\end{align}

(iii) The action of $\mathbb{Z}$ on $\mathbb{R}\times\mathbb{T}^2$ corresponding to (\ref{equation example}) is given by 
\[\sigma:\mathbb{R}\times\mathbb{T}^2\times\mathbb{Z}\rightarrow\mathbb{R}\times\mathbb{T}^2,\,\,\,((r,z_1,z_2),n)\mapsto(r+n,z_1,z_1^nz_2).
\]Moreover, this action is free. To verify its properness, we take two compact subsets $K$ and $L$ of $\mathbb{R}\times\mathbb{T}^2$. Then there exists $n_0\in\mathbb{N}$ such that the projection of $K$ onto $\mathbb{R}$ is contained in the interval $[-n_0,n_0]$. This implies that the set $\{n\in\mathbb{Z}:\,K\cap L.n\neq\emptyset\}$ is finite and thus that $\sigma$ is proper (cf. [Ne08b], Example 1.2.2 (a)). In view of the Quotient Theorem (cf. \cite[Kapitel VIII, Abschnitt 21]{To00}), we get a $\mathbb{Z}$-principal bundle
\[(\mathbb{R}\times\mathbb{T}^2,M,\mathbb{Z},\pr,\sigma),
\]where $M:=(\mathbb{R}\times\mathbb{T}^2)/\mathbb{Z}$ and $\pr:\mathbb{R}\times\mathbb{T}^2\rightarrow M$ denotes the canonical quotient map. In particular, we conclude from (i), (ii) and the previous discussion that $\Gamma\mathbb{A}$ is canonically isomorphic to $C^{\infty}(M)$. 


(iv) Next, one easily verifies that the map
\[\sigma':M\times\mathbb{T}\rightarrow M,\,\,\,([(r,z_1,z_2)],z)\mapsto [(r,z_1,z_2z)]
\]defines a smooth action of $\mathbb{T}$ on $M$, which is is free and proper. We thus get a $\mathbb{T}$-principal bundle $(M,M/\mathbb{T},\mathbb{T},\pr',\sigma')$, where $\pr':M\rightarrow M/\mathbb{T}$ denotes the canonical quotient map. In particular, Theorem \ref{NCT^nB for manifold again} (b) implies that the triple $(C^{\infty}(M),\mathbb{T},\alpha)$ is a smooth noncommutative principal bundle $\mathbb{T}$-bundle. 

(v) It remains to show that $M$ is non-trivial as principal $\mathbb{T}$-bundle over $M/\mathbb{T}\cong\mathbb{T}^2$. For this it is enough to show that $\pi_1(M)\ncong\mathbb{Z}^3$. Indeed, if $M\cong\mathbb{T}^3$, then $\pi_1(M)\cong\mathbb{Z}^3$. We proceed as follows: Let $\mathbb{R}\ltimes_S\mathbb{R}^2$ be the (left-) semidirect product of $\mathbb{R}$ and $\mathbb{R}^2$ defined by the homomorphism
\[S:\mathbb{R}\rightarrow\Aut(\mathbb{R}^2),\,\,\,S(r)(x,y):=(x,y+rx).
\]Then the assignment
\[(r,x,y).n:=(n,0,0)(r,x,y)=(r+n,x,y+rx),\,\,\,r,x,y\in\mathbb{R}, n\in\mathbb{Z},
\]is the lifting of the action $\sigma$ of part (iii) to (the universal covering) $\mathbb{R}^3$ ($\cong\mathbb{R}\ltimes_S\mathbb{R}^2$ as manifolds). In particular, the homogeneous space defined by the discrete subgroup $\mathbb{Z}\ltimes_S\mathbb{Z}^2$ of $\mathbb{R}\ltimes_S\mathbb{R}^2$ is diffeomorphic to $M$, i.e.,
\[(\mathbb{R}\ltimes_S\mathbb{R}^2)/(\mathbb{Z}\ltimes_S\mathbb{Z}^2)\cong M.
\]We thus conclude from the corresponding exact sequence of homotopy groups (cf. [Br93], Theorem VII.6.7 or [Ne08b], Theorem 6.3.17) that the map
\[\delta_1:\pi_1(M)\rightarrow\pi_0(\mathbb{Z}\ltimes_S\mathbb{Z}^2)\cong\mathbb{Z}\ltimes_S\mathbb{Z}^2,\,\,\,\delta_1([\gamma]):=(n,m,m'),
\]where $(n,m,m')\in\mathbb{Z}\ltimes_S\mathbb{Z}^2$ is such that $\widetilde{\gamma}(1)=\widetilde{\gamma}(0).(n,m,m')$ holds for a continuous lift \mbox{$\widetilde{\gamma}:[0,1]\rightarrow\mathbb{R}\ltimes_S\mathbb{R}^2$} of the loop $\gamma$ in $M$ with $\widetilde{\gamma}(0)=(0,0,0)$, is an isomorphism of groups. This proves the claim.
\end{proof}

\section{Outlook: Towards a Classification of Noncommutative Principal Torus Bundles}

In Section \ref{a geometric approach to NCPB} we have introduced the concept of (non-trivial) noncommutative principal torus bundles. Since we gave a complete classification of trivial noncommutative principal torus bundles in \cite[Section 4]{Wa11b}, it is, of course, a natural ambition to work out a classification theory for (non-trivial) noncommutative principal torus bundles. To be more precise, the problem is the following:

\begin{open problem}\label{open problem classification of NCP torus bundles}
(Classification of noncommutative principal torus bundles). Let $B$ be a unital locally convex algebra. Work out a good classification theory for noncommutative principal $\mathbb{T}^n$-bundles $(A,\mathbb{T}^n,\alpha)$ for which $A^{\mathbb{T}^n}=B$, i.e., classify all noncommutative principal $\mathbb{T}^n$-bundles $(A,\mathbb{T}^n,\alpha)$ for which $A^{\mathbb{T}^n}=B$.
\end{open problem}

As is well-know from classical differential geometry, the relation between locally and globally defined objects is important for many constructions and applications. For example, a (non-trivial) principal bundle $(P,M,G,q,\sigma)$ can be considered as a geometric object that is glued together from local pieces which are trivial, i.e., which are of the form $U\times G$ for some (arbitrary small) open subset $U$ of $M$. This approach immediately leads to the concept of $G$-valued cocycles and therefore to a cohomology theory, called the \v Cech cohomology of the pair $(M,G)$. This cohomology theory gives a complete classification of principal bundles with structure group $G$ and base manifold $M$ (cf. \cite[Kapitel IX.5]{To00}).



\begin{open problem}
(A noncommutative \v Cech cohomology). In view of the previous discussion, a possible approach to the Open Problem \ref{open problem classification of NCP torus bundles} is to work out a \emph{noncommutative version} of the \v Cech cohomology. For this, it is worth to study the paper \cite{Du96}. Moreover, ideas of the classification methods for loop algebras of \cite{ABP04} might be of particular interest. Also, a possible point of view for the relation between locally and globally defined objects in the noncommutative world can be found in \cite{CaMa00}; their subspaces are described by ideals and one glues algebras along ideals, using a pull-back construction. 
\end{open problem}



\begin{remark}
(Classification of locally trivial Hopf-Galois Extensions). A classification procedure for noncommutative principal torus bundles may as well lead to a classification theory of so-called \emph{locally trivial Hopf--Galois extensions} (cf. P.M. Hajac (editor): Quantum symmetries in noncommutative geometry, in preparation).
\end{remark}

\begin{appendix}

\section{The Projective Tensor Product}

This part of the appendix is devoted to some results concerning the projective tensor product of locally convex spaces. 

\begin{definition}\label{proj. tensor product III}
(The projective tensor product). Let $E$ and $F$ be two locally convex spaces. Further, let $E\odot F$ be the algebraic tensor product of $E$ and $F$. The \emph{projective tensor product} $E\otimes F$ of $E$ and $F$ is the algebraic tensor product $E\odot F$ equipped with the strongest locally convex topology for which the canonical ``projection"
\[\pi:E\times F\rightarrow E\odot F,\,\,\,(e,f)\mapsto e\odot f
\]is continuous. The corresponding completion of $E\otimes F$ with respect to this topology is denoted by $E\widehat{\otimes}F$. 
\end{definition}

\begin{remark}\label{universal property of projective tensor product topology}
(Universal property of the projective tensor product). We recall that the projective tensor product $E\otimes F$ of two locally convex spaces $E$ and $F$ has the following universal property: For every locally convex space $Z$ the canonical (algebraic) isomorphism of the space of bilinear maps of $E\times F$ into $Z$ onto the space of linear maps of $E\otimes F$ into $Z$ induces an (algebraic) isomorphism of the space of continuous bilinear maps of $E\times F$ into $Z$ onto the space of continuous linear maps of $E\otimes F$ into $Z$. In particular, a bilinear map $\varphi:E\times F\rightarrow Z$ is continuous if and only if there exists a continuous linear map $\psi:E\otimes F\rightarrow Z$ satisfying $\psi\circ\pi=\varphi$. A nice reference for these statements can be found in \cite[Part III]{Tre67}.
\end{remark}


\begin{lemma}\label{continuity of maps EoF to E'oF'}
Let $E,E',F$ and $F'$ be four locally convex spaces. Further, let $\alpha:E\rightarrow E'$ and $\beta:F\rightarrow F'$ be two continuous linear maps. Then the linear map
\[\alpha\otimes \beta:E\otimes F\rightarrow E'\otimes F',\,\,\,\alpha\otimes\beta(e\otimes f):=\alpha(e)\otimes\beta(f)
\]is continuous.
\end{lemma}

\begin{proof}
If $\pi':E'\times F'\rightarrow E'\otimes F',\,\,\,(e',f')\mapsto e'\otimes f'$ denotes the canonical ``projection", then the claim follows from the fact that $\pi'\circ(\alpha\times\beta)$ is continuous as a composition of continuous maps.
\end{proof}

\begin{proposition}\label{AoB as l.c. algebra}
Let $A$ and $B$ be two locally convex algebras. Then $A\otimes B$ becomes a locally convex algebra, when equipped with the multiplication
\[m:A\otimes B\times A\otimes B\rightarrow A\otimes B,\,\,\,m(a\otimes b, a'\otimes b'):=aa'\otimes bb'.
\]
\end{proposition}

\begin{proof}
(i) Suppose that the algebra structure on $A$ is given by the continuous bilinear map 
\[m_A:A\times A\rightarrow A,\,\,\,aa':=m_A(a,a')
\]and that the algebra structure on $B$ is given by the continuous bilinear map 
\[m_B:B\times B\rightarrow B,\,\,\,bb':=m_B(b,b').
\]Then Remark \ref{universal property of projective tensor product topology} implies that the map $m_A$, resp. $m_B$, induces a continuous linear 
\[m^A:A\otimes A\rightarrow A,\,\,\,a\otimes a'\mapsto aa',
\]resp.
\[m^B:B\otimes B\rightarrow B,\,\,\,b\otimes b'\mapsto bb'.
\]

(ii) To show that the map $m$ defines the structure of a locally convex algebra on $A\otimes B$, we use Remark \ref{universal property of projective tensor product topology} again: Indeed, the bilinear map $m$ is continuous if and only if the linear map 
\[n:A\otimes A\otimes B\otimes B\rightarrow A\otimes B,\,\,\,n(a\otimes a'\otimes b\otimes b'):=aa'\otimes bb'
\]is continuous. Here, we have used the canonical isomorphism 
\[A\otimes B\otimes A\otimes B\cong A\otimes A\otimes B\otimes B.
\]Now, a short observations shows that $n=m_A\otimes m_B$, and therefore Lemma \ref{continuity of maps EoF to E'oF'} implies that $m$ is continuous.
\end{proof}

\section{The Smooth Exponential Law}

In the last part of the appendix we discuss some useful results on the smooth exponential law which will be used several times in this paper.

For arbitrary sets $X$ and $Y$, let $Y^X$ be the set of all mappings from $X$ to $Y$. Then the following ``exponential law" holds: For any sets $X,Y$ and $Z$, we have
\[Z^{X\times Y}\cong(Z^Y)^X
\]as sets. To be more specific, the map
\[Z^{X\times Y}\cong(Z^Y)^X,\,\,\,f\rightarrow f^{\vee}
\]is a bijection, where
\[f^{\vee}:X\rightarrow Z^Y,\,\,\,f^{\vee}(x):=f(x,\cdot).
\]The inverse map is given by
\[(Z^Y)^X\rightarrow Z^{X\times Y},\,\,\,g\mapsto g^{\wedge},
\]where
\[g^{\wedge}:X\times Y\rightarrow Z,\,\,\,g^{\wedge}(x,y):=g(x)(y).
\]Next we consider an important topology for spaces of smooth functions:

\begin{definition}\label{smooth compact open topology}
(Smooth compact open topology). If $M$ is a locally convex manifold and $E$ a locally convex space, then the \emph{smooth compact open topology} on $C^{\infty}(M,E)$ is defined by the embedding
\[C^{\infty}(M,E)\hookrightarrow\prod_{n\in\mathbb{N}_0}C(T^nM,T^nE),\,\,\,f\mapsto(T^nf)_{n\in\mathbb{N}_0},
\]where the spaces $C(T^nM,T^nE)$ carry the compact open topology. Since $T^nE$ is a locally convex space isomorphic to $E^{2^n}$, the spaces $C(T^nM,T^nE)$ are locally convex and we thus obtain a locally convex topology on $C^{\infty}(M,E)$.
\end{definition}

It is natural to ask if there exists an ``exponential law"  for the space of smooth functions endowed with the smooth compact open topology, i.e., if we always have
\begin{align}
C^{\infty}(M\times N,E)\cong C^{\infty}(M,C^{\infty}(N,E))\label{exp law}
\end{align}
as a locally convex space, for all locally convex smooth manifolds $M$, $N$ and locally convex spaces $E$. In general, the answer is \emph{no}. Nevertheless, it can be shown that (\ref{exp law}) holds if $N$ is a finite-dimensional manifold over $\mathbb{R}$. To be more precise:

\begin{lemma}\label{smooth exp law}
Let $M$ be a real locally convex smooth manifold, $N$ a finite-dimensional smooth manifold and $E$ a topological 
vector space. Then a map $f:M\rightarrow C^{\infty}(N,E)$ is smooth if and only if the map
\[f^{\wedge}:M\times N\rightarrow E,\,\,\,f^{\wedge}(m,n):=f(m)(n)
\]is smooth. Moreover, the following map is an isomorphism of locally convex spaces:
\[C^{\infty}(M,C^{\infty}(N,E))\rightarrow C^{\infty}(M\times N,E),\,\,\,f\mapsto f^{\wedge}.
\]
\end{lemma}

\begin{proof}
A proof can be found in \cite[Appendix, Lemma A3]{NeWa07}.
\end{proof}

\end{appendix}

\affiliationone{
   Stefan Wagner\\
   Westf\"alische Wilhelms-Universit\"at M\"unster\\
   Mathematisches Institut\\
   Einsteinstr.~62, D-48149 M\"unster\\ Germany\\
                \email{swagn\_02@uni-muenster.de}}

\end{document}